\documentclass[a4paper,12pt,reqno]{article}
\usepackage{amsthm}
\usepackage{amsmath}
\usepackage{amssymb}
\usepackage[mathscr]{eucal}
\usepackage[latin2]{inputenc}
\usepackage[T1]{fontenc}
\usepackage{indentfirst}
\usepackage{amsfonts}
\usepackage{graphicx}
\usepackage{color}
\usepackage{verbatim}
\usepackage[top=2.5cm, bottom=2.5cm, left=2.5cm, right=2.5cm]{geometry}

\newtheoremstyle{new_plain}
	{}
	{}
	{\itshape}
	{}
	{\sffamily\bfseries}
	{.}
	{5pt plus1pt minus1pt\relax}
	{\thmnumber{#2. }\thmname{#1}\thmnote{ #3}}

\newtheoremstyle{nev_circle_definition}
	{}
	{}
	{\normalfont}
	{}
	{\sffamily\bfseries}
	{.}
	{5pt plus1pt minus1pt\relax}
	{\thmnumber{#2. }\thmname{#1}\thmnote{ #3}}
	
\theoremstyle{plain}
	\newtheorem{theorem}{Theorem}
	\newtheorem{lemma}[theorem]{Lemma}
	\newtheorem{proposition}[theorem]{Proposition}
	
	\newtheorem{conjecture}[theorem]{Conjecture}

\theoremstyle{definition}
	
	\newtheorem{remark}[theorem]{Remark}

\numberwithin{theorem}{section}
\numberwithin{equation}{section}

\title{\bfseries \Large Zeros of orthogonal polynomials near an algebraic singularity of the measure}

\author{\Large \'Arp\'ad Baricz\thanks{The research of \'A. Baricz was supported by the J\'anos Bolyai Research Scholarship of the Hungarian Academy of Sciences.}, Tivadar Danka\thanks{The research of T. Danka was supported by ERC Advanced Grant No. 267055 and by the \'UNKP-\'UNKP-16-3  New National Excellence Program of the Ministry of Human Capacities}}
\date{}

\begin{document}

\maketitle

\begin{abstract}
In this paper we study the local zero behavior of orthogonal polynomials around an algebraic singularity, that is, when the measure of orthogonality is supported on $ [-1,1] $ and behaves like $ h(x)|x - x_0|^\lambda dx $ for some $ x_0 \in (-1,1) $, where $ h(x) $ is strictly positive and analytic. We shall sharpen the theorem of Yoram Last and Barry Simon and show that the so-called fine zero spacing (which is known for  $ \lambda = 0$) unravels in the general case, and the asymptotic behavior of neighbouring zeros around the singularity can be described with the zeros of the function $ c J_{\frac{\lambda - 1}{2}}(x) + d J_{\frac{\lambda + 1}{2}}(x) $, where $ J_a(x) $ denotes the Bessel function of the first kind and order $ a $. Moreover, using Sturm-Liouville theory, we study the behavior of this linear combination of Bessel functions, thus providing estimates for the zeros in question.
\end{abstract}


\section{Introduction}

Let $ \mu $ be a finite Borel measure supported on the real line with finite power moments. Then there is an unique sequence of polynomials
\[
	p_n(x) = \gamma_n x^n + \dots, \quad n = 0, 1, 2, \dots
\]
such that $ \gamma_n > 0 $ and the orthogonality relation
\[
	\int p_n(x) p_m(x) d\mu(x) = \delta_{n,m}, \quad n, m = 0, 1, 2, \dots
\]
holds. $ p_n(x) $ is called the $ n $-th orthonormal polynomial with respect to $ \mu $. The monic orthogonal polynomials are denoted with $ \pi_n(x) = \gamma_n^{-1}p_n(x) $. \\

The zeros of such orthogonal polynomials has been an intensively studied subject since the beginning of the 20th century, and several breakthroughs have been made in the recent years. Recently, Eli Levin and Doron Lubinsky proved in \cite{Levin-Lubinsky} that if $ \mu $ is a finite Borel measure supported on $ [-1,1] $ and $ x_0 $ is an arbitrary point in $ (-1,1) $, then if $ \mu $ is regular in the sense of Stahl and Totik (for the definition see \cite{Stahl-Totik}) and $ \mu $ is absolutely continuous in a small neighborhood $ (x_0 - \varepsilon, x_0 + \varepsilon) $ of some $ x_0 \in (-1,1) $ with
\[
	d\mu(x) = w(x) dx, \quad x \in (x_0 - \varepsilon, x_0 + \varepsilon)
\]
there for a strictly positive and continuous weight $ w(x) $, then
\begin{equation}\label{Levin-Lubinsky_result}
	\lim_{n \to \infty} n (x_{k+1,n}(x_0) - x_{k,n}(x_0)) = \pi \sqrt{1 - x_0^2}, \quad k \in \mathbb{Z}
\end{equation}
holds, where the $ x_{k,n}(x_0) $-s denote the zeros of $ p_n(x) $ in ascending order centered around $ x_0 $ as
\[
	\cdots < x_{-2,n}(x_0) < x_{-1,n}(x_0) \leq x_0 < x_{1,n}(x_0) < x_{2,n}(x_0) < \cdots,
\]
and we write $ x_{0,n}(x_0) = x_0 $ for convenience, which may or may not be a zero. This behavior was termed \textit{clock zero spacing} by Barry Simon. Theorems of this type in such generality, where only local continuity and Stahl-Totik regularity is assumed, has been unprecedented. This was made possible by the universality results of Lubinsky \cite{Lubinsky_1} regarding the scaling limit of the Christoffel-Darboux kernel around $ x_0 $. The universality result and fine zero spacing were subsequently extended by Simon in \cite{Simon} and Totik in \cite{Totik_1} to measures supported on more general sets, although they used very different methods. In a slightly different setting, studying doubling measures on $ [-1,1] $, Mastroianni and Totik also proved in \cite{Mastroianni-Totik} that if $ \mu $ is doubling then there is a constant $ C $ independent of $ n $ such that
\begin{equation}\label{Mastroianni-Totik_spacing}
	C^{-1} \leq \frac{x_{k+1, n} - x_{k,n}}{\Delta_n(x_{k,n})} \leq C, \quad k = 0, 1, \ldots, n
\end{equation}
holds, where $ \Delta_n(t) = n^{-1} \sqrt{1 - t^2}  + n^{-2} $ and $ -1 < x_{1,n} < x_{2,n} < \cdots < x_{n,n} < 1 $ denotes the zeros of $ p_n(x) $ in ascending order. They also managed to show that in some sense, the converse is true, i.e. if (\ref{Mastroianni-Totik_spacing}) holds and the neighbouring Cotes numbers are bounded away from zero and infinity, then the measure is doubling on $ [-1,1] $. It was observed in \cite{Varga} by Varga that the results of Mastroianni and Totik also hold true locally, that is if we only assume the doubling property in a neighbourhood of some point in the support. \\

If, however, some explicit singular behavior of the weight is assumed, not much is known about the fine zero spacing. For weights with a jump singularity, A. Foulqui\'e Moreno, A. Mart\'inez-Finkelshtein and V. L. Sousa proved in \cite[Proposition 9]{Moreno-Finkelshtein-Sousa} that clock zero spacing does not hold around the jump, but besides this negative result, no precise assertions were made. For algebraic singularity of type $ |x - x_0|^\lambda $, Yoram Last and Barry Simon proved, see \cite[Theorem 8.5]{Last-Simon}, that if $ \mu $ is absolutely continuous in a neighbourhood of $ x_0 $ with
\[
	0 < \liminf_{x \to x_0} \frac{\mu^\prime(x)}{|x - x_0|^\lambda} \leq \limsup_{x \to x_0} \frac{\mu^\prime(x)}{|x - x_0|^\lambda} < \infty
\]
there for some $ \lambda > -1 $ there, then
\begin{equation}\label{Last-Simon_result}
	\limsup_{n \to \infty} n |x_{1,n}(x_0) - x_{-1,n}(x_0)| < \infty,
\end{equation}
but otherwise no more information about the zero spacing is known. This result was significantly sharpened by Varga in \cite{Varga}, where he assumed local doubling property in a neighbourhood of $ x_0 $, but it is still unknown what happens precisely when the weight has an algebraic singularity $ |x - x_0|^{\lambda} $ there. \\

The aim of this paper is to describe the zero spacing near the algebraic singularity in a sharper way. The zeros around the singularities are described with the aid of Bessel functions of the first kind. For each $ c, d \in \mathbb{R} $ and $ a > -1 $, the solutions of the equation
\begin{equation}\label{Bessel_equation}
	c x^{-a}J_{a}(x) + d x^{-a} J_{a+1}(x) = 0
\end{equation}
for $ x \in \mathbb{R} \setminus \{ 0 \} $ are denoted with
\begin{equation}\label{Bessel_equation_solutions}
	\cdots < j_{-2}(a, c, d) < j_{-1}(a, c, d) \leq 0 < j_1(a, c, d) < j_2(a, c, d) < \cdots,
\end{equation}
and we write $ j_0(a, c, d) = 0 $ for convenience, which may or may not be a solution. Note that since $ J_a(z) = z^a G_a(z) $, where $ G_a(z) $ is an entire function, (\ref{Bessel_equation}) makes sense for negative arguments. Since the Bessel functions are continuous and the positive zeros of $ J_{a}(x) $ and $ J_{a + 1}(x) $ interlace in a way such that the $ k $-th zero of $ J_{a}(x) $ is always smaller than the $ k $-th zero of $ J_{a + 1}(x) $ (for this fact see \cite[15.22]{Watson}), we can indeed write (\ref{Bessel_equation_solutions}). Note that $ j_k(a, 0, 1) $ is just the $ k $-th zero of $ x^{-a} J_{a+1}(x) $ and $ j_k(a, 1, 0) $ is the $ k $-th zero of $ x^{-a} J_{a}(x) $. Our main theorem is the following.

\begin{theorem}\label{main_theorem_1}
Let $ \mu $ be a finite Borel measure supported on $ [-1,1] $ defined by
\begin{equation}\label{mu_definition}
	d\mu(x) = h(x) (1-x)^{\alpha}(1 + x)^{\beta} \prod_{\nu = 1}^{n_0} |x - x_\nu|^{\lambda_\nu} dx, \quad x \in [-1,1],
\end{equation}
where $ h $ is a positive analytic function and $ \alpha, \beta, \lambda_1, \dots, \lambda_{n_0} > -1 $. Let
\begin{equation}\label{centered_zeros}
	\cdots < x_{-2,n}(x_\nu) < x_{-1,n}(x_\nu) \leq x_\nu < x_{1,n}(x_\nu) < x_{2,n}(x_\nu) < \cdots
\end{equation}
be the zeros of the $ n $-th orthonormal polynomial (with the additional notation $ x_{0,n}(x_\nu) = x_\nu $, which may or may not be a zero) centered around the algebraic singularity $ x_\nu $, which is located in the interior of $ [-1,1] $, i.e. $ x_\nu \in (-1,1) $. \\

(a) Let $ k \in \mathbb{Z} $ be fixed. If $ \arccos x_\nu $ is a rational multiple of $ \pi $, say $ \arccos x_\nu = \pi \frac{p}{q} $ where $ \gcd(p,q) = 1 $, then there are $ q $ distinct constants $ c_0, \dots, c_{q-1} \in \mathbb{R} $ and $ q $ distinct constants $ d_0, \dots, d_{q-1} \in \mathbb{R} $ such that for $ n_l = l q + m $, $ m = 0, 1, \dots, q-1 $, we have
\[
	\lim_{l \to \infty} \frac{n_l}{\sqrt{1 - x_\nu^2}} (x_{k+1,n_l}(x_\nu) - x_{k,n_l}(x_\nu)) = j_{k+1}\Big(\frac{\lambda_\nu-1}{2}, c_m, d_m\Big) - j_{k}\Big(\frac{\lambda_\nu-1}{2}, c_m, d_m\Big),
\]
where $ j_k(a,c,d) $ denotes the zeros of \eqref{Bessel_equation} ordered as \eqref{Bessel_equation_solutions}. Moreover, the constants $ c_l, d_l $ for two distinct $ k_0, k_1 \in \mathbb{Z} $ are equal, if the sign of $ k_0 $ and $ k_1 $ equals. \\

(b) If $ \arccos x_\nu $ is not a rational multiple of $ \pi $, then for any fixed $ k \in \mathbb{Z} $ and any constants $ c, d \in \mathbb{R} $ there is a subsequence $ n_l $ such that
\[
	\lim_{l \to \infty} \frac{n_l}{\sqrt{1 - x_\nu^2}}(x_{k+1,n_l}(x_\nu) - x_{k,n_l}(x_\nu)) = j_{k + 1}\Big(\frac{\lambda_\nu-1}{2}, c, d\Big) - j_{k}\Big(\frac{\lambda_\nu-1}{2}, c, d\Big).
\]
\end{theorem}

Since $ J_{\frac{1}{2}}(z) = \sqrt{\frac{2}{\pi z}} \sin(z) $ and $ J_{-\frac{1}{2}}(z) = \sqrt{\frac{2}{\pi z}} \cos(z) $, in the special case $ \lambda_\nu = 0 $, Theorem \ref{main_theorem_1} gives back the result of Levin and Lubinsky. Theorem \ref{main_theorem_1} not only says that the fine zero spacing fails spectacularly, it also makes very precise assertions about the limiting behavior of zeros. As far as we are aware, there are no similar results known in the literature. Theorem \ref{main_theorem_1} provides an explicit example where fine zero spacing like (\ref{Levin-Lubinsky_result}) unravels, moreover it also shows that the quantitative behavior of the zeros depend heavily on the location of the algebraic singularity $ x_\nu $, which is also a newly observed phenomenon. Although the proof of Theorem \ref{main_theorem_1} works only when $ \mu $ is supported on $ [-1,1] $ and defined by (\ref{mu_definition}), we conjecture that the following holds in more general situations.

\begin{conjecture}\label{main_conjecture}
Let $ \mu $ be a finite Borel measure supported on a compact subset $ K $ of the real line and assume that $ \mu $ is regular in the sense of Stahl and Totik on $ K $. Let $ x_0 \in \operatorname{int}(K) $ and suppose that $ \mu $ is absolutely continuous in a neighbourhood $ (x_0 - \varepsilon, x_0 + \varepsilon) $ of $ x_0 $ and
\[
	d\mu(x) = w(x) |x - x_0|^\lambda dx, \quad x \in (x_0 - \varepsilon, x_0 + \varepsilon)
\]
holds there for some $ \lambda > -1 $ and a strictly positive and continuous weight $ w(x) $. \\

(a) Let $ k \in \mathbb{Z} $ be fixed. If $ \nu_K([x_0, \infty)) $ is a rational number, where $ \nu_K $ denotes the equilibrium measure of $ K $, say $ \nu_K([x_0, \infty)) = \frac{p}{q} $, where $ \gcd(p,q) = 1 $, then there are $ q $ distinct constants $ c_0, \dots, c_{q-1} \in \mathbb{R} $  and $ q $ distinct constants $ d_0, \dots, d_{q-1} \in \mathbb{R} $ such that for $ n_l = lq+m $, $ m = 0, 1, \dots, q-1 $, we have
\[
	\lim_{l \to \infty} n_l \pi \omega_K(x_0)(x_{k+1,n_l}(x_0) - x_{k,n_l}(x_0)) = j_{k+1}\Big(\frac{\lambda-1}{2}, c_m, d_m\Big) -  j_{k}\Big(\frac{\lambda-1}{2}, c_m, d_m\Big),
\]
where $ \omega_K(x) $ denotes the equilibrium density of $ K $, and $ j_k(a,c,d) $ denotes the zeros of \eqref{Bessel_equation} ordered as \eqref{Bessel_equation_solutions}. Moreover, the constants $ c_l, d_l $ for two distinct $ k_0, k_1 \in \mathbb{Z} $ are equal, if the sign of $ k_0 $ and $ k_1 $ equals. \\

(b) If $ \nu_K([x_0, \infty)) $ is not a rational number, then for any $ c, d \in \mathbb{R} $, there is a subsequence $ n_l $ such that
\[
	\lim_{l \to \infty} n_l \pi \omega_K(x_0) (x_{k+1,n_l}(x_0) - x_{k,n_l}(x_0)) = j_{k+1}\Big(\frac{\lambda-1}{2}, c, d\Big) -  j_{k}\Big(\frac{\lambda-1}{2}, c, d\Big).
\]
\end{conjecture}

For the notion of Stahl-Totik regularity see \cite{Stahl-Totik}, and for the notion of equilibrium measure and other potential theoretic concepts, see the book \cite{Ransford}. Theorem \ref{main_theorem_1} is a special case of Conjecture \ref{main_conjecture}, since the equilibrium density of the interval $ [-1,1] $ is $ \omega_{[-1,1]}(x) = \pi^{-1} (1 - x^2)^{-\frac{1}{2}} $, which is the well known Chebyshev distribution. \\

We also study the zeros of the special function
\begin{equation}\label{psi_acd}
	\psi_{a,c,d}(x) = c J_a(x) + d J_{a+1}(x), \quad a > -1, \quad c,d \in \mathbb{R}.
\end{equation}
The roots of the equation $ \psi_{a,c,d}(x) = 0 $ coincide with $ j_k(a,c,d) $ for positive $ k $. Using Sturm-Liouville theory, we managed to obtain estimates of neighbouring zeros.

\begin{theorem}\label{zero_convexity_theorem} Suppose that $ c $ and $ d $ are nonzero real numbers. \\
(i) If $ a \geq \frac{1}{2} $ and $ c $ has the same sign as $ d $, then
\begin{equation}\label{zero_convexity_1}
	j_{k+2}(a,c,d) - j_{k+1}(a,c,d) < j_{k+1}(a,c,d) - j_{k}(a,c,d),
\end{equation}
for all $ k \geq 1 $. \\
(ii) If $ 0 < a < \frac{1}{2} $ and $ c $ has the same sign as $ d $, then
\begin{equation}\label{zero_convexity_2}
	\frac{\left(1+\frac{\sigma}{j_{k+2}(a,c,d)}\right)\left(1+\frac{\sigma}{j_{k}(a,c,d)}\right)}{\left(1+\frac{\sigma}{j_{k+1}(a,c,d)}\right)^2}
>e^{\sigma\Delta^2 j_k(a,c,d)-\sigma^2\Delta^2 j_k^2(a,c,d)},
\end{equation}
for all $ k \geq 1 $, where
\[
	\sigma=\frac{(2a+1)cd}{c^2+d^2}
\]
and
\[
	\Delta^2j_k(a,c,d)=\frac{1}{j_{k+2}(a,c,d)}+\frac{1}{j_{k}(a,c,d)}-\frac{2}{j_{k+1}(a,c,d)}.
\]
(iii) If $ -\frac{1}{2} < a \leq 0 $ and $ c $ has the same sign as $ d $, then
\begin{equation}\label{zero_convexity_3}
	\frac{\left(1+\frac{\sigma}{j_{k+2}(a,c,d)}\right)\left(1+\frac{\sigma}{j_{k}(a,c,d)}\right)}{\left(1+\frac{\sigma}{j_{k+1}(a,c,d)}\right)^2}
<e^{\sigma\Delta^2 j_k(a,c,d)},
\end{equation}
for all $ k \geq 1 $, where $ \sigma $ and $ \Delta^2j_k(a,c,d) $ are defined as before. \\
(iv) If $ -1 < a < -\frac{1}{2} $, $ d^2 \geq c^2 $ and $ c $ has the opposite sign as $ d $, then
\begin{equation}\label{zero_convexity_4}
	\frac{(\sigma+j_{k+2}(a,c,d))(\sigma+j_{k}(a,c,d))}{(\sigma+j_{k+1}(a,c,d))^2}<e^{\frac{1}{\sigma}\delta^2j_{k}(a,c,d)},
\end{equation}
where
\[
	\delta^2 j_k(a,c,d)={j_{k+2}(a,c,d)}-2{j_{k+1}(a,c,d)}+j_k(a,c,d),
\]
and $ \sigma $ is as defined before.
\end{theorem}

\begin{theorem}\label{zero_comparison_theorem}
Suppose that $ c $ and $ d $ are nonzero real numbers. \\
(i) If $ a \geq \frac{1}{2} $ and $ c $ has the same sign as $ d $, then
\begin{equation}
	 j_{k+1}(a,c,d) - j_{k}(a,c,d) > \pi
\end{equation}
for all $ k \geq 1 $. \\
(ii) If $a\geq0,$ $d^2\geq c^2,$ and $c$ has the same sign as $d,$ then
\begin{equation}\label{zero_comparison_case_2}
	\sigma\ln\frac{j_{k+1}(a,c,d)}{j_k(a,c,d)}+{j_{k+1}(a,c,d)}-{j_k(a,c,d)}>\pi
\end{equation}
for all $k\geq1.$ \\
(iii) If $ -1 < a < -\frac{1}{2} $, $ d^2 \geq c^2 $ and $ c $ has the opposite sign as $ d $, then
\begin{equation}\label{zero_comparison_case_3}
	j_{k+1}(a,c,d)-j_{k}(a,c,d)-\sigma\ln\frac{\sigma+j_{k+1}(a,c,d)}{\sigma+j_{k}(a,c,d)}<\pi.
\end{equation}
\end{theorem}

\begin{remark}
The estimates in Theorems \ref{zero_convexity_theorem} and \ref{zero_comparison_theorem} were obtained by finding an appropriate second order linear differential equation, then applying the Liouville transform of the form $ z^\prime(x) = x^{p} ((2a+1)cd+(c^2+d^2)x)^{q} $ for $ p, q \in \mathbb{Z} $. The estimates are a result of using the Sturm comparison and convexity theorems to this transformed equation for specific values of $ p $ and $ q $. For details on the Sturm comparison and convexity theorems, see \cite[Theorem 1]{Deano-Gil-Segura}. If $ c = 0 $ or $ d = 0 $, we obtain the results known for Bessel functions, see \cite[Theorem 21]{Deano-Gil-Segura}.
\end{remark}

Since $ j_k(a,c,d)  $ is between the $ k $-th zero of $ J_a(x) $ and $ J_{a+1}(x) $ (which is implied by the interlacing zero property of $ J_a(x) $ and $ J_{a+1}(x) $, see \cite[15.22]{Watson}) and the $ k $-th zero of $ J_\nu(x) $ is $ \big(k + \frac{2\nu - 1}{4}\big) \pi + o(1) $, see \cite[15.53]{Watson}, we have $ \frac{j_{k+1}(a,c,d)}{j_k(a,c,d)} = 1 + o(1) $, hence \eqref{zero_comparison_case_2} can be written as
\[
	j_{k+1}(a,c,d) - j_k(a,c,d) > \pi + o(1)
\]
and \eqref{zero_comparison_case_3} can be written as
\[
	j_{k+1}(a,c,d) - j_k(a,c,d) < \pi + o(1),
\]
which is asymptotically sharp. These kind of asymptotic inequalities are available for the full range of parameters. In this case, we can state the following theorem.

\begin{theorem}\label{Bessel_zero_asymptotic_theorem}
Let $ a > -1 $ and $ c,d \in \mathbb{R} $. Then
\begin{equation}\label{Bessel_zero_asymptotic_equation}
	\lim_{k \to \infty} \big(j_{k+1}(a,c,d) - j_k(a,c,d)\big) = \pi.
\end{equation}
\end{theorem}

Theorem \ref{Bessel_zero_asymptotic_theorem} says that although the distance of neighbouring zeros may oscillate, these oscillations are smoothed out at infinity. \\

The outline of the paper is the following. In Section \ref{section_asymptotics} we employ the Riemann-Hilbert analysis and steepest descent method of Deift and Zhou, see for example \cite{Deift}, to obtain strong asymptotic formulas for the monic orthogonal polynomials $ \pi_n(x) $. Then, in Section \ref{section_zeros}, we extract the information about the zeros relative to the singularity after scaling the asymptotic formulas we obtained, which will prove Theorem \ref{main_theorem_1}. Finally in Section \ref{Bessel_linear_combination_zeros} we study the positive zeros $ cJ_a(x) + dJ_{a+1}(x) $, where we prove Theorems \ref{zero_convexity_theorem}, \ref{zero_comparison_theorem}, \ref{Bessel_zero_asymptotic_theorem}.

\section{Asymptotics for orthogonal polynomials}\label{section_asymptotics}

\subsection{Riemann-Hilbert problem for the orthogonal polynomials}

Let $ -1 = x_{0} < x_1 < \dots < x_{n_0} < x_{n_0 + 1} = 1 $ be given points in the interval $ [-1,1] $ and define the measure $ \mu $ by
\[
	d\mu(x) = (1-x)^{\alpha}(1 + x)^{\beta} h(x) \prod_{\nu = 1}^{n_0} |x - x_\nu|^{\lambda_\nu} dx, \quad x \in [-1,1],
\]
where $ h $ is a positive analytic function and $ \alpha, \beta, \lambda_1, \dots, \lambda_{n_0} > -1 $, $ \lambda_1, \dots, \lambda_{n_0} \neq 0 $. In this section our goal is to obtain asymptotic formulas for orthogonal polynomials with respect to $ \mu $. To achieve this, we shall closely follow the Riemann-Hilbert analysis of M. Vanlessen \cite{Vanlessen} and A. B. J. Kuijlaars, M. Vanlessen \cite{Kuijlaars-Vanlessen}. First we define the Riemann-Hilbert problem for the $ 2 \times 2 $ matrix valued function $ Y(z) =  (Y_{ij}(z))_{i,j=1}^{2} $ as in \cite{Vanlessen}. Suppose that \\
(a) $ Y(z) $ is analytic for $ z \in \mathbb{C} \setminus [-1,1] $. \\
(b) For all $ x \in (-1,1)\setminus \{ x_1, \dots, x_{n_0} \} $ the limits
\[
	Y_+(x) = \lim_{\substack{z \to x \\ \operatorname{Im}(z) > 0}} Y(z), \quad Y_-(x) = \lim_{\substack{z \to x \\ \operatorname{Im}(z) < 0}} Y(z)
\]
exists and the jump condition
\[
	Y_{+}(x) = Y_{-}(x) \begin{pmatrix} 1 & w(x) \\ 0 & 1 \end{pmatrix}, x \in (-1,1) \setminus \{ x_1, \dots, x_{n_0} \}
\]
holds. \\
(c) For the behavior of $ Y(z) $ near infinity we have
\[
	Y(z) = (I + O(z^{-1})) \begin{pmatrix} z^n & 0 \\ 0 & z^{-n} \end{pmatrix}, \quad z \to \infty.
\]
(d) For the behavior of $ Y(z) $ near $ z = 1 $ we have
\[
	Y(z) =
	\begin{cases}
		O \begin{pmatrix} 1 & |z-1|^\alpha \\ 1 & |z-1|^\alpha \end{pmatrix}, & \text{if } \alpha < 0, \\ 
		O \begin{pmatrix} 1 & \log|z-1| \\ 1 & \log|z-1| \end{pmatrix}, & \text{if } \alpha = 0, \\ 
		O \begin{pmatrix} 1 & 1 \\ 1 & 1 \end{pmatrix}, & \text{if } \alpha > 0, \\ 
	\end{cases}
\]
as $ z \to 1 $, $ z \in \mathbb{C} \setminus [-1,1] $. \\
(e) For the behavior of $ Y(z) $ near $ z = -1 $ we have
\[
	Y(z) =
	\begin{cases}
		O \begin{pmatrix} 1 & |z+1|^\beta \\ 1 & |z+1|^\beta \end{pmatrix}, & \text{if } \beta < 0, \\ 
		O \begin{pmatrix} 1 & \log|z+1| \\ 1 & \log|z+1| \end{pmatrix}, & \text{if } \beta = 0, \\ 
		O \begin{pmatrix} 1 & 1 \\ 1 & 1 \end{pmatrix}, & \text{if } \beta > 0, \\ 
	\end{cases}
\]
as $ z \to -1 $, $ z \in \mathbb{C} \setminus [-1,1] $. \\
(f) For the behavior of $ Y(z) $ near the singularity $ x_\nu $, $ \nu \in \{ 1, \dots, n_0 \} $ we have
\[
	Y(z) =
	\begin{cases}
		O \begin{pmatrix} 1 & |z-x_\nu|^{\lambda_\nu} \\ 1 & |z-x_\nu|^{\lambda_\nu} \end{pmatrix}, & \text{if } \lambda_\nu < 0, \\ 
		O \begin{pmatrix} 1 & 1 \\ 1 & 1 \end{pmatrix}, & \text{if } \lambda_\nu > 0, \\ 
	\end{cases}
\]
as $ z \to x_\nu$, $ z \in \mathbb{C} \setminus [-1,1] $. The unique solution for this Riemann-Hilbert problem can be expressed in terms of orthogonal polynomials. If $ \pi_n(z) $ denotes the monic orthogonal polynomial of degree $ n $ with respect to the measure $ \mu $ and  $ \gamma_n(\mu) = \gamma_n $ denotes the leading coefficient of the orthonormal polynomial $ p_n(z) $, then, see \cite[Theorem 2.2]{Vanlessen}, $ Y(z) $ takes the form
\begin{equation}\label{model_case_Y_OP}
	Y(z) =
	\begin{pmatrix}
	\pi_n(z) & \frac{1}{2\pi i} \int_{-1}^{1} \frac{\pi_n(x)}{x - z} d\mu(x) \\
	-2\pi i \gamma_{n-1}^{2} \pi_{n - 1}(z) & - \gamma_{n-1}^{2} \int_{-1}^{1} \frac{\pi_{n - 1}(x)}{x - z} d\mu(x)
	\end{pmatrix}.
\end{equation}

To give an asymptotic formula for $ Y(z) $, we have to use a series of transformations $ Y \mapsto T \mapsto S \mapsto R $, thus obtaining a $ 2 \times 2 $ matrix valued function $ R(z) $ which is close to the identity matrix $ I $. These transformations are rather complicated and to carry out these, we shall need a few special functions. First define
\begin{equation}\label{phi_def}
	\varphi(z) = z + \sqrt{z^2 - 1},
\end{equation}
which is the conformal mapping from the exterior of the interval $ [-1,1] $ to the exterior of the unit disk. Next, we have to extend the weight function $ w(x) $ to the complex plane. This is done as in \cite[Section 3.2]{Vanlessen}. The factor $ h $ is analytic, therefore there is a neighborhood $ U $ of $ [-1,1] $ such that $ h $ is analytic in $ U $. For the analytic continuation of the factors $ |x-x_\nu|^{\lambda\nu} $, we have to proceed differently. For each singularity $ x_\nu, \nu = 1, \dots, n_0 $ there is a hyperbola $ \Gamma_{x_\nu} $ going through $ x_\nu $ such that the images of $ \Gamma_{x_\nu} \cap \mathbb{C}_+ $ and $ \Gamma_{x_\nu} \cap \mathbb{C}_- $ under the mapping $ \varphi $ are straight rays. The contour $ \Gamma_{x_\nu} $ divides $ \mathbb{C} $ into a left and right part denoted with $ K_{x_\nu}^{l} $ and $ K_{x_\nu}^{r} $. For convenience, we introduce the notation $ K_{x_0}^{r} = K_{x_{n_0 + 1}}^{l} = \mathbb{C} $. With these, we define $ w(z) $ in $ U \setminus ((-\infty, -1] \cup [1, \infty) \cup [\cup_{\nu=1}^{n_0} \Gamma_{x_\nu}]) $ as
\[
	w(z) = (1-z)^{\alpha}(1+z)^{\beta} h(z) \prod_{k=1}^{\nu} (z-x_k)^{\lambda_k} \prod_{k=\nu + 1}^{n_0} (x_k - z)^{\lambda_k}, \quad z \in K_{\nu}^{r} \cap K_{\nu+1}^{l}.
\]
For $ \nu = 1, \dots, n_0 $ we define the auxiliary function $ W_{x_\nu}(z) $ as
\begin{align*}
	W_{x_\nu}(z) =  & (1-z)^{\frac{\alpha}{2}}(z+1)^{\frac{\beta}{2}} h(z)^{\frac{1}{2}} \prod_{k=1}^{\nu-1} (z-x_\nu)^{\frac{\lambda_k}{2}} \prod_{k=\nu + 1}^{n_0} (x_k - z)^{\frac{\lambda_k}{2}} \\
	& \times \begin{cases} (z-x_\nu)^{\frac{\lambda_\nu}{2}} & \text{if } z \in (K_{x_\nu}^{l} \cap U) \setminus \mathbb{R}, \\ (x_\nu-z)^{\frac{\lambda_\nu}{2}} & \text{if } z \in (K_{x_\nu}^{r} \cap U) \setminus \mathbb{R}. \end{cases}
\end{align*} 
Note that we have
\[
	W_{x_\nu, +}(x) W_{x_\nu, -}(x) = w(x), \quad x \in (x_{\nu-1}, x_{\nu + 1}) \setminus \{ x_\nu \}
\]
and
\begin{equation}\label{W_x_nu_jumps}
	W_{x_\nu}^{2}(z) =
	\begin{cases}
		w(z) e^{-\pi i \lambda_\nu} & \text{if } z \in (K_{x_\nu}^{r} \cap \mathbb{C}_+) \cup (K_{x_\nu}^{l} \cap \mathbb{C}_-) \\
		w(z) e^{\pi i \lambda_\nu} & \text{if } z \in (K_{x_\nu}^{l} \cap \mathbb{C}_+) \cup (K_{x_\nu}^{r} \cap \mathbb{C}_-).
	\end{cases}
\end{equation}
The auxiliary functions $ f_{x_\nu}(z) $ for $ \nu = 1, \dots, n_0 $ are defined as
\begin{equation}\label{f_x_nu_definition}
	f_{x_\nu}(z) =
	\begin{cases}
		i \log \varphi(z) - i \log \varphi_+(x_\nu) & \textnormal{if } \operatorname{Im}(z) > 0, \\
		- i \log \varphi(z) - i \log \varphi_+(x_\nu) & \textnormal{if } \operatorname{Im}(z) < 0,
	\end{cases}
\end{equation}
where $ \varphi_+(x) = \lim_{z \to x, \operatorname{Im}(z) > 0} \varphi(z) = \exp(i \arccos x) $. \\

Because of the singularities at $ -1 = x_0 < x_1 < \cdots < x_{n_0} < x_{n_0 + 1} = 1 $, the strong asymptotics are different in their proximity than away from them. Suppose that $ U_{\delta, x_0}, U_{\delta, x_1}, \dots, U_{\delta,x_{n_0 + 1}} $ is a disjoint collection of open disks with radius $ \delta $ around the $ x_\nu $-s and let $ \Sigma $ be a lens-shaped contour around $ [-1,1] $ closing at the singularities defined such that the preimage of the contour parts in $ U_{\delta, x_\nu} $ under the mapping $ f_{x_\nu}(z) $ are the rays $ \{ z: \arg(z-x_\nu)= \frac{\pi}{4} + k\frac{\pi}{2} \} $ for $ k = 0, 1, 2, 3 $. For more on the choice of the contour, see \cite[Section 4.2]{Vanlessen}. We shall give asymptotic formulas inside the reduced contour $ \Sigma_R = [\cup_{\nu=0}^{n_0+1} \partial {U_{\delta, x_\nu}}] \cup [\Sigma \setminus \cup_{\nu=0}^{n_0 + 1} U_{\delta, x_\nu}] $.

\subsection{Parametrix for the outside region}

To construct the asymptotic formula for $ Y(z) $, we need the solution $ N(z) $ of a model Riemann-Hilbert problem. $ N(z) $ is defined as
\begin{equation}
	N(z) = D_{\infty}^{\sigma_3} \begin{pmatrix} \frac{a(z) + a(z)^{-1}}{2} & \frac{a(z) - a(z)^{-1}}{2i} \\ \frac{a(z) - a(z)^{-1}}{-2i} & \frac{a(z) + a(z)^{-1}}{2} \end{pmatrix} D(z)^{-\sigma_3},
\end{equation}
where $ a(z) = \frac{(z-1)^{\frac{1}{4}}}{(z + 1)^{\frac{1}{4}}} $ and $ D(z) $ is the Szeg\H{o} function (associated to the measure $ \mu $) which is defined in $ \mathbb{C} \setminus [-1,1] $ by
\[
	D(z) = \frac{(z-1)^{\frac{\alpha}{2}} (z+1)^{\frac{\beta}{2}} \prod_{\nu=1}^{n_0} (z-x_\nu)^{\frac{\lambda}{2}}}{\varphi(z)^{\frac{1}{2}(\alpha + \beta + \sum_{\nu=1}^{n_0} \lambda_\nu )}} \exp\bigg( \frac{\sqrt{z^2-1}}{2\pi} \int_{-1}^{1} \frac{\log h(x)}{\sqrt{1-x^2}} \frac{dx}{z-x} \bigg),
\]
moreover $ \sigma_3 $ denotes the Pauli matrix
\[
	\sigma_3 = \begin{pmatrix} 1 & 0 \\ 0 & -1 \end{pmatrix}
\]
and the symbol $ h(z)^{\sigma_3} $ denotes the $ 2 \times 2 $ matrix
\[
	h(z)^{\sigma_3} = \begin{pmatrix} h(z) & 0 \\ 0 & h(z)^{-1} \end{pmatrix}.
\]
For a complete discussion about $ N(z) $ and its role in the solution, see \cite[Section 3.2]{Vanlessen}. The boundary values of the Szeg\H{o} function $ D(z) $ along $ \mathbb{R} $ can be expressed as
\begin{equation}\label{D_boundary_value}
	D_+(x) = \sqrt{w(x)} e^{-i\psi_\nu(x)}, \quad x \in (x_\nu, x_{\nu + 1}),
\end{equation}
where
\begin{equation}\label{psi_def}
\begin{aligned}
	\psi_\nu(x) = \frac{1}{2} \bigg[ \bigg( \alpha + \beta + \sum_{k=1}^{n_0} \lambda_k \bigg) \arccos x & - \bigg( \alpha + \sum_{k=\nu +1}^{n_0} \lambda_k \bigg) \pi \\
	& + \frac{\sqrt{1-x^2}}{\pi} \int_{-1}^{1} \frac{\log h(t)}{\sqrt{1-t^2}} \frac{dt}{t-x} \bigg].
\end{aligned}
\end{equation}
Note that $ \psi_\nu(x) + \frac{\pi}{2} \lambda_\nu = \psi_{\nu - 1}(x) $. Regarding the boundary values of the functions $ \frac{a(z) + a(z)^{-1}}{2} $ and $ \frac{a(z) - a(z)^{-1}}{2i} $, we have
\begin{equation}\label{m_boundary_value}
\begin{aligned}
	\frac{a_+(x) + a_{+}^{-1}(x)}{2} & = \frac{e^{-i\frac{\pi}{4}}}{\sqrt{2}(1-x^2)^{\frac{1}{4}}} \varphi_+(x)^{\frac{1}{2}}, \\
	\frac{a_+(x) - a_{+}^{-1}(x)}{2i} & = \frac{e^{i\frac{\pi}{4}}}{\sqrt{2}(1-x^2)^{\frac{1}{4}}} \varphi_+(x)^{-\frac{1}{2}},
\end{aligned}
\end{equation}
where $ \varphi(z) $ is defined by \eqref{phi_def} and its boundary value is $ \varphi_+(x) = \exp(i \arccos x) $.

\subsection{Parametrix for the singularities}

To give an asymptotic formula for $ Y(z) $ around the singularities $	-1 = x_0 < x_1 < \dots < x_{n_0} < x_{n_0 + 1} = 1 $, we need some parametrix functions $ P_{x_\nu}(z) $. These are constructed from Bessel and modified Bessel functions. The constructon depends on whether $ x_\nu $ is an endpoint of $ [-1, 1] $ or it is in its interior. First we give the parametrix for the endpoints. These were constructed in \cite{Kuijlaars-McLaughlin-VanAssche-Vanlessen}. Define the function $ \Psi_{\alpha}^{e}(z) $ as
\begin{equation}
\Psi_{\alpha}^{e}(z) =
		\begin{pmatrix}
		\frac{1}{2} H_{\alpha}^{(1)}(2\sqrt{-z}) & \frac{1}{2} H_{\alpha}^{(2)}(2\sqrt{-z}) \\
		\pi \sqrt{z} (H_{\alpha}^{(1)})^{\prime}(2\sqrt{-z}) & \pi \sqrt{z} (H_{\alpha}^{(2)})^{\prime}(2\sqrt{-z}) 
		\end{pmatrix} e^{\frac{1}{2}\alpha \pi i \sigma_3},
		\quad\frac{2\pi}{3} < \arg z < \pi,
\end{equation}
where $ H_{\alpha}^{(1)} $ and $ H_{\alpha}^{(2)} $ denotes the Hankel functions of the first and second kind. Although $ \Psi_{\alpha}^{e}(z) $ can be defined in the whole complex plane, this shall be enough for our purposes. For more details, see \cite{Kuijlaars-McLaughlin-VanAssche-Vanlessen} and \cite{Kuijlaars-Vanlessen}. This way, the parametrix $ P_1(z) $ can be defined as
\[
	P_1(z) = N(z) W_1(z)^{\sigma_3} \frac{1}{\sqrt{2}} \begin{pmatrix} 1 & -i \\ -i & 1 \end{pmatrix} f_1(z)^{\frac{\sigma_3}{4}} (2\pi n)^{\frac{\sigma_3}{2}} \Psi_{\alpha}^{e}(n^2 f_1(z)) W_1(z)^{-\sigma_3} \varphi(z)^{-n\sigma_3}.
\]

Now let $ \nu \in \{ 1, \dots, n_0 \} $ be fixed. We divide the complex plane around $ x_\nu $ into eight congruent octants defined as
\[
	O_i = \bigg\{ z:  \frac{(k-1)\pi}{4} \leq \arg(z-x_\nu) \leq \frac{k \pi}{4} \bigg\}, \quad i = 1, 2, \dots, 8.
\]
Define a $ 2 \times 2 $ matrix valued function $ \Psi_{\alpha}^{s}(z) $ in the first and the fourth octant $ O_1 $ and $ O_4 $ as
\begin{equation}
	\Psi_{\alpha}^{s}(z) =
	\begin{cases}
		\frac{1}{2} \sqrt{\pi z} \begin{pmatrix} e^{-i\frac{2\alpha + 1}{4}\pi} H_{\frac{\alpha + 1}{2}}^{(2)}(z) & -i e^{i\frac{2\alpha + 1}{4}\pi} H_{\frac{\alpha + 1}{2}}^{(1)}(z) \\ e^{-i\frac{2\alpha + 1}{4}\pi} H_{\frac{\alpha - 1}{2}}^{(2)}(z) & -i e^{i\frac{2\alpha + 1}{4}\pi} H_{\frac{\alpha - 1}{2}}^{(1)}(z) \end{pmatrix} & \text{if } z \in O_1 \\
		\frac{1}{2} \sqrt{\pi(-z)} \begin{pmatrix} ie^{i\frac{2\alpha + 1}{4}\pi }H_{\frac{\alpha + 1}{2}}^{(1)}(-z) & -e^{-i\frac{2\alpha + 1}{4}\pi }H_{\frac{\alpha + 1}{2}}^{(2)}(-z) \\ -ie^{i\frac{2\alpha + 1}{4}\pi}H_{\frac{\alpha - 1}{2}}^{(1)}(-z) & e^{-i\frac{2\alpha + 1}{4}\pi}H_{\frac{\alpha - 1}{2}}^{(2)}(-z)  \end{pmatrix} & \text{if } z \in O_4.
	\end{cases}
\end{equation}
The definition of $ \Psi_{x_\nu}(z) $ can be extended to the whole complex plane, but for our purposes, $ O_1 $ and $ O_4 $ are enough. For the complete definition, see \cite[Section 4.2]{Vanlessen}. Using these, the parametrix, as originally given in \cite[Section 4.3]{Vanlessen}, is defined as
\begin{equation}\label{parametrix_at_singularity_right}
\begin{aligned}
	P_{x_\nu}(z) = & N(z) W_{x_\nu}(z)^{\sigma_3} e^{\frac{1}{4} \lambda_\nu \pi i \sigma_3} \varphi_+(x_\nu)^{n\sigma_3} e^{-i\frac{\pi}{4}\sigma_3} \\
	& \times \frac{1}{\sqrt{2}} \begin{pmatrix} 1 & i \\ i & 1 \end{pmatrix}  \Psi_{\lambda_\nu}^{s}(n f_{x_\nu}(z)) W_{x_\nu}(z)^{-\sigma_3} \varphi(z)^{-n\sigma_3} 
\end{aligned}
\end{equation}
if $ z \in O_1 $ and
\begin{equation}\label{parametrix_at_singularity_left}
\begin{aligned}
	P_{x_\nu}(z) = & N(z) W_{x_\nu}(z)^{\sigma_3} e^{-\frac{1}{4} \lambda_\nu \pi i \sigma_3} \varphi_+(x_\nu)^{n\sigma_3} e^{-i\frac{\pi}{4}\sigma_3} \\
	& \times \frac{1}{\sqrt{2}} \begin{pmatrix} 1 & i \\ i & 1 \end{pmatrix}  \Psi_{\lambda_\nu}^{s}(n f_{x_\nu}(z)) W_{x_\nu}(z)^{-\sigma_3} \varphi(z)^{-n\sigma_3} 
\end{aligned}
\end{equation}
if $ z \in O_4 $, where the only difference is the sign of the exponent of the third term.

\subsection{Asymptotic formulas}

By applying the transformations $ Y \mapsto T \mapsto S \mapsto R $ detailed in \cite{Vanlessen}, we obtain a $ 2 \times 2 $ matrix valued function $ R(z) $ which is uniformly close to the identity matrix, i.e. $ R(z) = I + O(n^{-1}) $ for all $ z \in \mathbb{C} \setminus \Sigma_R $. This can be expressed as
\begin{equation}\label{R_product_form}
\begin{aligned}
	R(z) = 2^{n \sigma_3} Y(z) & \varphi(z)^{-n\sigma_3} \begin{pmatrix} 1 & 0 \\ -w(z)^{-1} \varphi(z)^{-2n} & 1 \end{pmatrix} \\
	& \times \begin{cases}
	N(z)^{-1} & \text{if } z \in \mathbb{C}_+ \setminus (\Sigma \cup [\cup_{\nu=0}^{n_0 + 1} U_{\delta, x_\nu}]) \\
	P_{x_\nu}(z)^{-1} & \text{if } z \in \mathbb{C}_+ \cap U_{\delta, x_\nu}.
	\end{cases}
\end{aligned}
\end{equation}
From this, by unraveling the transformations, we can obtain formulas for $ \pi_n(x) $. In order to establish asymptotic formulas for $ \pi_n(x) $, $ x \in [-1,1] $, we have to distinguish three cases. \\
(i) $ x $ is far away from the singularities and the endpoints, that is $ x \in (x_\nu + \delta, x_{\nu+1} - \delta) $, $ \nu = 1, \dots, n_0 $. \\
(ii) $ x $ is near some endpoint, that is $ x \in (-1, -1 + \delta) \cup (1 - \delta, 1) $. \\
(iii) $ x $ is near some singularity, that is $ x \in (x_\nu - \delta, x_\nu + \delta) $, $ \nu = 1, \dots, n_0 $. \\

In the first and the second cases (although in a slightly different setting), Kuijlaars, McLaughlin, Van Assche and Vanlessen \cite{Kuijlaars-McLaughlin-VanAssche-Vanlessen} obtained the formula
\begin{equation}\label{asymptotics_away_from_singularities}
\begin{aligned}
	\pi_n(x) & = \frac{D_\infty}{2^n (1-x^2)^{\frac{1}{4}}} \sqrt{\frac{2}{w(x)}} \\
	& \times \Big[ \Big( 1 + O(n^{-1}) \Big) \cos \Big( \Big(n + \frac{1}{2}\Big) \arccos x + \psi_{\nu}(x) - \frac{\pi}{4} \Big) \\
	& \phantom{aaaaaa.}+ O(n^{-1}) \cos \Big( \Big(n - \frac{1}{2}\Big) \arccos x + \psi_{\nu}(x) - \frac{\pi}{4} \Big) \Big]
\end{aligned}
\end{equation}
for $ x \in (x_\nu + \delta, x_{\nu + 1} - \delta) $, $ \nu = 1, \dots, n_0 $ and
\begin{equation}\label{asymptotics_at_endpoint}
\begin{aligned}
	\pi_n(x) = & \frac{D_\infty}{2^n (1-x^2)^{\frac{1}{4}}} \sqrt{\frac{\pi (n \arccos x)}{w(x)}} \\
	& \times \Big[ \Big( 1 + O(n^{-1}) \Big) \Big( \cos \Big( \frac{1}{2} \arccos x + \psi_{n_0}(x) + \frac{\alpha \pi}{2} \Big) J_{\alpha}(n \arccos x) \\
	& \phantom{aaaaaaaaaaaaaa} + \sin \Big( \frac{1}{2} \arccos x + \psi_{n_0}(x) + \frac{\alpha \pi}{2} \Big) J_{\alpha}^{\prime}(n \arccos x) \Big) \\
	& \phantom{aa} + O(n^{-1}) \Big( \cos \Big( -\frac{1}{2} \arccos x + \psi_{n_0}(x) + \frac{\alpha \pi}{2} \Big) J_{\alpha}(n \arccos x) \\
	& \phantom{aaaaaaaa.a} + \sin \Big( - \frac{1}{2} \arccos x + \psi_{n_0}(x) + \frac{\alpha \pi}{2} \Big) J_{\alpha}^{\prime}(n \arccos x) \Big) \Big]
\end{aligned}
\end{equation}
for $ x \in (1-\delta, 1) $, where the $ O(n^{-1}) $ terms are uniform in $ x $. (Note that the weight $ w(x) $ and the phases $ \psi_\nu $ in (\ref{asymptotics_at_endpoint}) are not the same as in \cite{Kuijlaars-McLaughlin-VanAssche-Vanlessen}, but everything in their calculations carries through verbatim to our case.) \\

The third case is partially dealt with in \cite{Vanlessen}. Although the Riemann-Hilbert analysis was carried out near the singularites, an analogous asymptotic formula was not stated. The rest of this section is devoted to find this. We closely follow the lines of \cite{Kuijlaars-Vanlessen}.

\begin{lemma}
Around the singularity $ x_\nu $, $ \nu = 1, \dots, n_0 $ the first column of $ Y(x) $ has the form
\begin{equation}\label{Y_first_column_left_to_singularity}
\begin{aligned}
	\begin{pmatrix} Y_{11}(x) \\ Y_{12}(x) \end{pmatrix} = & \sqrt{\frac{\pi n(\arccos x - \arccos x_\nu)}{2 w(x)}} 2^{-n \sigma_3} M_+^l(x) \\ & \times \begin{pmatrix} e^{i \frac{3\pi}{4}} J_{\frac{\lambda_\nu + 1}{2}}(n (\arccos x - \arccos x_\nu)) \\ e^{-i \frac{\pi}{4}} J_{\frac{\lambda_\nu - 1}{2}}(n (\arccos x - \arccos x_\nu)) \end{pmatrix}
\end{aligned}
\end{equation}
for $ x \in (x_\nu - \delta, x_\nu) $ and
\begin{equation}\label{Y_first_column_right_to_singularity}
\begin{aligned}
	\begin{pmatrix} Y_{11}(x) \\ Y_{12}(x) \end{pmatrix} = &  \sqrt{\frac{\pi n(\arccos x - \arccos x_\nu)}{2 w(x)}} 2^{-n \sigma_3} M_+^r(x)\\ & \times \begin{pmatrix} e^{-i\frac{\pi}{4}} J_{\frac{\lambda_\nu + 1}{2}}(n (\arccos x_\nu - \arccos x)) \\ e^{-i\frac{\pi}{4}} J_{\frac{\lambda_\nu - 1}{2}}(n (\arccos x_\nu - \arccos x)) \end{pmatrix}
\end{aligned}
\end{equation}
for $ x \in (x_\nu, x_\nu + \delta) $, where $ M(z) $ is defined by
\[
	M^l(z) = R(z) N(z) W_{x_\nu }(z)^{\sigma_3} e^{-i \frac{\lambda_\nu}{4} \pi \sigma_3} \varphi_+(x_\nu)^{n \sigma_3}  e^{-i\frac{\pi}{4}\sigma_3} \begin{pmatrix} 1 & i \\ i & 1 \end{pmatrix}.
\]
and
\[
	M^r(z) = R(z) N(z) W_{x_\nu }(z)^{\sigma_3} e^{i \frac{\lambda_\nu}{4} \pi \sigma_3} \varphi_+(x_\nu)^{n \sigma_3}  e^{-i\frac{\pi}{4}\sigma_3} \begin{pmatrix} 1 & i \\ i & 1 \end{pmatrix}.
\]
\end{lemma}
\begin{proof}
First we suppose that $ z \in U_{\delta, x_\nu} \cap K_{x_\nu}^{l} $, that is, we are in the left side of the singularity. According to (\ref{parametrix_at_singularity_left}) and (\ref{R_product_form}), $ Y(z) $ can be written as
\begin{align*}
	Y(z) = & 2^{-n\sigma_3} R(z) N(z) W_{x_\nu}(z)^{\sigma_3} e^{-\frac{\lambda_\nu}{4}\pi i \sigma_3} \varphi_+(x_\nu)^{n \sigma_3}  e^{-i\frac{\pi}{4}\sigma_3} \frac{1}{\sqrt{2}} \begin{pmatrix} 1 & i \\ i & 1 \end{pmatrix} \Psi_{\lambda_\nu}^{s}(n f_{x_\nu}(z)) \\
	& \times W_{x_\nu}(z)^{-\sigma_3} \varphi(z)^{-n \sigma_3} \begin{pmatrix} 1 & 0 \\ w(z)^{-1} \varphi(z)^{-2n} & 1 \end{pmatrix} \varphi(z)^{n \sigma_3}.
\end{align*}
Since $ W_{x_\nu}(z) = e^{i\frac{\pi}{2}\lambda_\nu} w(z)^{\frac{1}{2}} $ if $ z \in K_{x_\nu}^{l} \cap U_{\delta, x_\nu} $, the last terms simplify to
\[
	W_{x_\nu}(z)^{-\sigma_3} \varphi(z)^{-n \sigma_3} \begin{pmatrix} 1 & 0 \\ w(z)^{-1} \varphi(z)^{-2n} & 1 \end{pmatrix} \varphi(z)^{n \sigma_3} = \begin{pmatrix} e^{-i \frac{\pi}{2}\lambda_\nu} w(z)^{-\frac{1}{2}} & 0 \\ e^{i \frac{\pi}{2} \lambda_\nu} w(z)^{-\frac{1}{2}} & e^{i \frac{\pi}{2} \lambda_\nu} w(z)^{\frac{1}{2}}. \end{pmatrix}
\]
It follows that
\[
	\begin{pmatrix} Y_{11}(z) \\ Y_{12}(z) \end{pmatrix} = 2^{-n\sigma_3} R(z) E_{n, x_\nu}(z) \Psi_{\lambda_\nu}^{s}(n f_{x_\nu}(z)) \frac{1}{\sqrt{w(z)}} \begin{pmatrix} e^{-i \frac{\pi}{2} \lambda_\nu} \\ e^{i \frac{\pi}{2} \lambda_\nu} \end{pmatrix},
\]
where $ E_{n, x_\nu}(z) $ is defined as
\[
	E_{n,x_\nu}(z) = N(z) W_{x_\nu}(z)^{\sigma_3} e^{-\frac{\lambda_\nu}{4}\pi i \sigma_3} \varphi_+(x_\nu)^{n \sigma_3}  e^{-i\frac{\pi}{4}\sigma_3} \frac{1}{\sqrt{2}} \begin{pmatrix} 1 & i \\ i & 1 \end{pmatrix}.
\]
Now we aim to express $ \Psi_{\lambda_\nu}^{s}(n f_{x_\nu}(z)) \begin{pmatrix} e^{-i \frac{\pi}{2} \lambda_\nu} \\ e^{i \frac{\pi}{2} \lambda_\nu} \end{pmatrix} $ in terms of Bessel functions. Since $ H_{\gamma}^{(1)}(z) + H_{\gamma}^{(2)} = 2 J_{\gamma}(z) $, we have
\[
	\Psi_{\lambda_\nu}^{s}(z) \begin{pmatrix} e^{-i \frac{\pi}{2} \lambda_\nu} \\ e^{i \frac{\pi}{2} \lambda_\nu} \end{pmatrix} = \sqrt{\pi(-z)} \begin{pmatrix} e^{i \frac{3\pi}{4}} J_{\frac{\lambda_\nu + 1}{2}}(-z) \\ e^{-i \frac{\pi}{4}} J_{\frac{\lambda_\nu - 1}{2}}(-z) \end{pmatrix}.
\]
Now we let $ z \to x $ with $ x \in (x_\nu - \delta, x_\nu) $ and $ \operatorname{Im}(z) > 0 $. (\ref{f_x_nu_definition}) and $ \varphi_+(x) = \exp(i \arccos x) $ implies that $ (f_{x_\nu})_+(x) = \arccos x_\nu - \arccos x $. Therefore, we have
\begin{align*}
	\begin{pmatrix} Y_{11}(x) \\ Y_{12}(x) \end{pmatrix} = & \sqrt{\frac{\pi n}{2 w(x)}} 2^{-n \sigma_3} R_{+}(x) N_+(x) (W_{x_\nu })_+(x)^{\sigma_3} e^{-i \frac{\lambda_\nu}{4} \pi \sigma_3} \varphi_+(x_\nu)^{n\sigma_3}  e^{-i\frac{\pi}{4}\sigma_3} \begin{pmatrix} 1 & i \\ i & 1 \end{pmatrix} \\
	& \times (\arccos x - \arccos x_\nu)^{\frac{1}{2}} \begin{pmatrix} e^{i \frac{3\pi}{4}} J_{\frac{\lambda_\nu + 1}{2}}(n (\arccos x - \arccos x_\nu)) \\ e^{-i \frac{\pi}{4}} J_{\frac{\lambda_\nu - 1}{2}}(n (\arccos x - \arccos x_\nu)) \end{pmatrix}
\end{align*}
for all $ x \in (x_\nu - \delta, x_\nu) $, and this is precisely (\ref{Y_first_column_left_to_singularity}). (\ref{Y_first_column_right_to_singularity}) can be obtained with similar calculations.
\end{proof}

\begin{proposition}\label{proposition_strong_asymptotics_near_singularity}
For all $ \nu = 1, \dots, n_0 $,
\begin{equation}\label{OP_asymptotic_left_to_singularity}
\begin{aligned}
\pi_n(x) = & \frac{D_\infty}{2^n (1-x^2)^{\frac{1}{4}}} \sqrt{\frac{\pi n (\arccos x - \arccos x_\nu)}{w(x)}}  \\
	& \times \Big[ \Big( 1 + O(n^{-1}) \Big) \Big( \cos(\xi_{\nu, 1}(x) + \pi \lambda_\nu) J_{\frac{\lambda_\nu - 1}{2}}(n(\arccos x - \arccos x_\nu)) \\
	& \phantom{aaaaaaaaaaaaaaaa} - \sin(\xi_{\nu, 1}(x) + \pi \lambda_\nu) J_{\frac{\lambda_\nu + 1}{2}}(n(\arccos x - \arccos x_\nu)) \Big) \\
	& \phantom{aaaaaa} + O(n^{-1}) \Big( \cos(\xi_{\nu, 2}(x) + \pi \lambda_\nu) J_{\frac{\lambda_\nu - 1}{2}}(n(\arccos x - \arccos x_\nu)) \\ 
	& \phantom{aaaaaaaaaaaaaaaa} - \sin(\xi_{\nu, 2}(x) + \pi \lambda_\nu)J_{\frac{\lambda_\nu + 1}{2}}(n(\arccos x - \arccos x_\nu)) \Big) \Big]
\end{aligned}
\end{equation}
holds for $ x \in (x_\nu - \delta, x_\nu) $ and
\begin{equation}\label{OP_asymptotic_right_to_singularity}
\begin{aligned}
\pi_n(x) = & \frac{D_\infty}{2^n (1 - x^2)^{\frac{1}{4}}} \sqrt{\frac{\pi n (\arccos x_\nu - \arccos x)}{w(x)}} \\
	& \times \Big[ \Big( 1 + O(n^{-1}) \Big) \Big( \cos(\xi_{\nu ,1}(x)) J_{\frac{\lambda_\nu - 1}{2}}(n(\arccos x_\nu - \arccos x)) \\
	& \phantom{aaaaaaaaaaaaaaaa} + \sin(\xi_{\nu,1}(x)) J_{\frac{\lambda_\nu + 1}{2}}(n(\arccos x_\nu - \arccos x)) \Big) \\
	& \phantom{aaaaaa} + O(n^{-1}) \Big( \cos(\xi_{\nu,2}(x)) J_{\frac{\lambda_\nu - 1}{2}}(n(\arccos x_\nu - \arccos x)) \\ 
	& \phantom{aaaaaaaaaaaaaaaa} + \sin(\xi_{\nu,2}(x))J_{\frac{\lambda_\nu + 1}{2}}(n(\arccos x_\nu - \arccos x)) \Big) \Big]
\end{aligned}
\end{equation}
holds for $ x \in (x_\nu, x_\nu + \delta) $, where
\begin{equation}
	\xi_{\nu,k}(x) = \psi_{\nu}(x) - \frac{\pi}{4} \lambda_\nu - \frac{\pi}{4} + n \arccos x_\nu + (-1)^{k + 1} \frac{1}{2} \arccos x
\end{equation}
for $ k = 1, 2 $, and $ \psi_\nu(x) $ is defined by (\ref{psi_def}). Moreover, the error terms $ O(n^{-1}) $ hold uniformly for $ x \in (x_\nu - \delta, x_\nu] $ and $ x \in [x_\nu, x_\nu + \delta) $ respectively.
\end{proposition}
\begin{proof}
First assume that $ x \in (x_\nu - \delta, x_\nu) $. In order to show (\ref{OP_asymptotic_left_to_singularity}), we have to simplify $ M_+^l(x) $ in (\ref{Y_first_column_left_to_singularity}). According to (\ref{D_boundary_value}) and (\ref{m_boundary_value}), we have
\begin{equation}
\begin{aligned}
	N_+(x) = & D_{\infty}^{\sigma_3} \frac{e^{-i \frac{\pi}{4}}}{\sqrt{2}(1-x^2)^{\frac{1}{4}}} \begin{pmatrix} e^{i \frac{1}{2}\arccos x} & i e^{-i \frac{1}{2}\arccos x}\\ -i e^{-i \frac{1}{2}\arccos x} & e^{i \frac{1}{2}\arccos x} \end{pmatrix} \\
	& \times (\sqrt{w(x)} e^{-i\psi_{\nu - 1}(x)})^{-\sigma_3}.
\end{aligned}
\end{equation}
Because of (\ref{W_x_nu_jumps}), we have $ (D_+(x))^{-\sigma_3} (W_{x_\nu})_+(x)^{\sigma_3} = (e^{i \psi_{\nu - 1}(x) + i \frac{\pi}{2}\lambda_\nu})^{\sigma_3} $, therefore 
\begin{align*}
	N_+(x) & (W_{x_\nu })_+(x)^{\sigma_3} e^{-i \frac{\lambda_\nu}{4} \pi \sigma_3} \varphi_+(x_\nu)^{n \sigma_3} e^{-i \frac{\pi}{4}\sigma_3} \begin{pmatrix} 1 & i \\ i & 1 \end{pmatrix} \\
	& = D_\infty^{\sigma_3} \frac{e^{-i\frac{\pi}{4}}}{\sqrt{2}(1-x^2)^{\frac{1}{4}}} \begin{pmatrix} e^{i \xi_{\nu, 1}(x) + \pi \lambda_\nu} & ie^{-i \xi_{\nu, 1}(x) + \pi \lambda_\nu} \\ -i e^{i \xi_{\nu, 2}(x) + \pi \lambda_\nu} & e^{-i \eta_{\xi, 2}(x) + \pi \lambda_\nu} \end{pmatrix} \begin{pmatrix} 1 & i \\ i & 1 \end{pmatrix} \\
	& = D_\infty^{\sigma_3} \frac{e^{-i\frac{\pi}{4}}}{\sqrt{2}(1-x^2)^{\frac{1}{4}}} \begin{pmatrix} 2i \sin(\xi_{\nu, 1}(x) + \pi \lambda_\nu) & 2i \cos(\xi_{\nu, 1}(x) + \pi \lambda_\nu) \\ 2 \sin(\xi_{\nu, 2}(x) + \pi \lambda_\nu) & 2 \cos(\xi_{\nu, 2}(x) + \pi \lambda_\nu) \end{pmatrix}.
\end{align*}
Substituting this into (\ref{Y_first_column_left_to_singularity}), we obtain
\begin{align*}
\begin{pmatrix} Y_{11}(x) \\ Y_{12}(x) \end{pmatrix} & = \sqrt{\frac{\pi n}{w(x) }} \frac{(\arccos x - \arccos x_\nu)^{\frac{1}{2}}}{(1-x^2)^{\frac{1}{4}}} 2^{-n \sigma_3} R_+(x) D_\infty^{\sigma_3} \\
	& \times \begin{pmatrix}
		i \sin (\xi_{\nu, 1}(x) + \pi \lambda_\nu) & i \cos (\xi_{\nu, 1}(x) + \pi \lambda_\nu) \\
		\sin (\xi_{\nu, 2}(x) + \pi \lambda_\nu) & \cos (\xi_{\nu, 2}(x) + \pi \lambda_\nu) \\
	\end{pmatrix} \\
	& \times \begin{pmatrix}
		i J_{\frac{\lambda_\nu + 1}{2}}(n (\arccos x - \arccos x_\nu)) \\
		-i J_{\frac{\lambda_\nu - 1}{2}}(n (\arccos x - \arccos x_\nu))
	\end{pmatrix},
\end{align*}
which, by using $ R(z) = I + O(n^{-1}) $, where $ O(n^{-1}) $ is uniform in a small neighbourhood of $ x_\nu $ (see \cite[(3.30)]{Vanlessen}), gives (\ref{OP_asymptotic_left_to_singularity}). The formula (\ref{OP_asymptotic_right_to_singularity}) is obtained similarly.
\end{proof}

\section{Zeros around the algebraic singularity}\label{section_zeros}

In this section we study the zero spacing of the orthogonal polynomials near an algebraic singularity. First with the help of Proposition \ref{proposition_strong_asymptotics_near_singularity} we shall understand how the scaled polynomial $ \pi_n(x_0 + an^{-1}) $ behaves, then we shall obtain asymptotic formulas for the zeros around the singularity.

\begin{lemma}\label{lemma_scaling_behavior}
Let $ \xi \in (-1,1) $ and let $ a \in \mathbb{R} \setminus \{ 0 \} $ be fixed. Define $ a_n = \xi + a n^{-1} $ and $ \widetilde{a}_n = n (\arccos \xi - \arccos a_n) $. Then
\begin{equation}\label{a_tilde_n_behavior}
	\widetilde{a}_n = \frac{a}{\sqrt{1 - \xi^2}} + O\bigg( \frac{a^2}{n} \bigg)
\end{equation}
and
\begin{equation}\label{J_at_a_tilde_n_behavior}
	J_\lambda(|\widetilde{a}_n|) = J_\lambda \bigg( \frac{|a|}{\sqrt{1 - \xi^2}} \bigg) + O\bigg( \frac{|a|^{\lambda + 1}}{n} \bigg)
\end{equation}
holds.
\end{lemma}
\begin{proof}
(\ref{a_tilde_n_behavior}) follows immediately from the Taylor expansion of $ \arccos x $ around $ \xi $. (\ref{J_at_a_tilde_n_behavior}) follows from the fact that $ J_\lambda(z) = z^{\lambda} G_\lambda(z) $, where $ G_\lambda(z) $ is an entire function.
\end{proof}

\begin{lemma}
Let $ a > 0 $. With the notations of Proposition \ref{proposition_strong_asymptotics_near_singularity}, we have
\begin{equation}\label{scaled_asymptotics_positive}
\begin{aligned}
	\pi_n\Big(x_\nu + \frac{a}{n}\Big) & = \frac{n^{\frac{\lambda_\nu}{2}}\pi D_\infty}{2^n \sqrt{(1 - x_\nu^2) h_\nu(x_\nu)}}  \\
	& \phantom{aaa} \times \Big[ a^{-\frac{\lambda_\nu - 1}{2}} \cos(\xi_{\nu, 1}(x_\nu)) J_{\frac{\lambda_\nu - 1}{2}}\big( a (1 - x_\nu^2)^{-\frac{1}{2}}\big) \\
	& \phantom{aaaaaaaa} + a^{-\frac{\lambda_\nu - 1}{2}} \sin(\xi_{\nu, 1}(x_\nu)) J_{\frac{\lambda_\nu + 1}{2}}\big(a (1 - x_\nu^2)^{-\frac{1}{2}}\big) + O(n^{-1}) \Big]
\end{aligned}
\end{equation}
and
\begin{equation}\label{scaled_asymptotics_negative}
\begin{aligned}
	\pi_n&\Big(x_\nu - \frac{a}{n}\Big) = \frac{n^{\frac{\lambda_\nu}{2}}\pi D_\infty}{2^n \sqrt{(1 - x_\nu^2) h_\nu(x_\nu)}} \\
	& \phantom{aaa} \times \Big[ a^{-\frac{\lambda_\nu - 1}{2}} \cos(\xi_{\nu, 1}(x_\nu) + \pi \lambda_\nu) J_{\frac{\lambda_\nu - 1}{2}}\big(a (1 - x_\nu^2)^{-\frac{1}{2}}\big) \\
	& \phantom{aaaaaaaa} + a^{-\frac{\lambda_\nu - 1}{2}} \sin(\xi_{\nu, 1}(x_\nu) + \pi \lambda_\nu) J_{\frac{\lambda_\nu + 1}{2}}\big(a (1 - x_\nu^2)^{-\frac{1}{2}}\big) + O(n^{-1}) \Big],
\end{aligned}
\end{equation}
where $ h_\nu(x) $ is defined by
\begin{equation}\label{h_nu}
	h_\nu(x) = h(x) (1 - x)^\alpha (1 + x)^\beta \prod_{\substack{k=1 \\ k \neq \nu}}^{n_0} |x - x_k|^{\lambda_k}
\end{equation}
and the error term $ O(n^{-1}) $ is uniform for $ a $ in compact subsets of the real line.
\end{lemma}
\begin{proof}
According to the notations of Lemma \ref{lemma_scaling_behavior}, let $ a > 0 $ and define $ a_n = x_\nu + \frac{a}{n} $ with $ \widetilde{a}_n = n(\arccos x_\nu - \arccos a_n) $. For convenience, we shall write $ w(x) = h_\nu(x) |x - x_\nu|^{\lambda_\nu} $, where $ h_\nu(x) $ is defined by (\ref{h_nu}) and it is smooth at $ x_\nu $.

First notice that if $ f(x) $ is any smooth function, then $ f(a_n) = f(x_\nu) + O(a n^{-1}) $, which follows easily from the Taylor expansion around $ x_\nu $. According to this and (\ref{a_tilde_n_behavior}), the first term of (\ref{OP_asymptotic_right_to_singularity}) can be written as
\begin{equation}\label{scaled_asymptotics_1}
	\frac{D_\infty}{2^n (1 - a_n^2)^{\frac{1}{4}}} \sqrt{\frac{\pi \widetilde{a}_n}{w(a_n)}} = (1 + o(1)) \frac{n^{\frac{\lambda_\nu}{2}} \pi D_\infty a^{-\frac{\lambda_\nu - 1}{2}}}{2^n \sqrt{(1 - x_\nu^2) h_\nu(x_\nu)} }.
\end{equation}
Now (\ref{a_tilde_n_behavior}) and (\ref{J_at_a_tilde_n_behavior}), part of the second term which involves trigonometric functions and Bessel functions can be written as
\begin{equation}\label{scaled_asymptotics_2}
	\cos(\xi_{\nu , k}(a_n)) J_{\frac{\lambda_\nu - 1}{2}}(\widetilde{a}_n) = \cos(\xi_{\nu, k}(x_\nu)) J_{\frac{\lambda_\nu - 1}{2}}\big( a (1 - x_\nu^2)^{-\frac{1}{2}}\big) + O\bigg( \frac{a^{\frac{\lambda_\nu + 1}{2}}}{n} \bigg)
\end{equation}
and
\begin{equation}\label{scaled_asymptotics_3}
	\sin(\xi_{\nu , k}(a_n)) J_{\frac{\lambda_\nu + 1}{2}}(\widetilde{a}_n) = \sin(\xi_{\nu, k}(x_\nu)) J_{\frac{\lambda_\nu + 1}{2}}\big( a (1 - x_\nu^2)^{-\frac{1}{2}}\big) + O\bigg( \frac{a^{\frac{\lambda_\nu + 3}{2}}}{n} \bigg),
\end{equation}
where $ k = 1, 2 $. Note that since $ \lambda_\nu > -1 $, the error terms are uniform for $ a $ in compact subsets of $ [0,\infty) $. Using that $ a^{-\lambda} J_\lambda(a) $ is bounded for $ a $ in compact neighborhoods of $ 0 $, which is implied by the fact that $ J_\lambda(z) = z^\lambda G(z) $ where $ G(z) $ is an entire function (in fact, it is bounded by $ (2^{\lambda} \Gamma(\lambda+1))^{-1} $, see \cite[Corollary]{Minakshisundaram-Szasz}), (\ref{scaled_asymptotics_1}) - (\ref{scaled_asymptotics_3}) gives (\ref{scaled_asymptotics_positive}), which is what we wanted to prove. (\ref{scaled_asymptotics_negative}) can be obtained similarly.
\end{proof}

Now, as implied by (\ref{scaled_asymptotics_positive}), the $ k $-th scaled zero 
\begin{equation}\label{scaled_zeros}
	a_{k,n}(x_\nu) = \frac{n}{\sqrt{1 - x_\nu^2}} (x_{k,n}(x_\nu) - x_\nu)
\end{equation}
to the right of $ x_\nu $ is the $ k $-th solution of the equation
\begin{equation}\label{zero_equation}
\begin{aligned}
	\cos(n& \arccos x_\nu + \varphi_\nu) a^{- \frac{\lambda_\nu - 1}{2}} J_{\frac{\lambda_\nu - 1}{2}}(a) \\ & + \sin(n \arccos x_\nu + \varphi_\nu) a^{- \frac{\lambda_\nu - 1}{2}} J_{\frac{\lambda_\nu + 1}{2}}(a) + O(n^{-1}) = 0,
\end{aligned}
\end{equation}
where the phase $ \varphi_\nu $ is defined by
\[
	\varphi_\nu = \psi_\nu(x_\nu) - (1 + \lambda_\nu)\frac{\pi}{4} + \frac{1}{2} \arccos x_\nu.
\]

Before proceeding with the study of the zeros, we need a lemma about the simplicity of the zeros of $ \psi_{a,c,d}(x) = cJ_{a}(x) + dJ_{a+1}(x) $.

\begin{lemma}\label{Bessel_equation_simple_roots_lemma}
Let $ c, d \in \mathbb{R} $ and assume that $ a > -1 $. Then the solutions of the equation
\begin{equation}\label{Bessel_model_equation}
	 cJ_{a}(x) + dJ_{a+1}(x) = 0
\end{equation}
are real and simple. (Recall that these solutions were denoted with $ j_k(a, c, d) $ for all $ k > 1 $ according to the conventions in (\ref{Bessel_equation_solutions}) and below.)
\end{lemma}
\begin{proof}
The case when either $ c $ or $ d $ are zero is well-known, so we do not lose from generality if we suppose that $ c, d \neq 0 $. We shall prove that when $ c, d \neq 0 $ and $ a > -1 $, the function
\[
	\psi_{a, c, d}(x) = c J_{a}(x) + d J_{a + 1}(x)
\]
has only real and simple zeros. In view of the corresponding recurrence relation $ xJ_{\nu}^\prime(x)=\nu J_{\nu}(x)-xJ_{\nu+1}(x) $, this statement is equivalent to the following: when $ c, d \neq 0 $ and $ a > -1 $, the function $ x \mapsto (ad + cx) J_a(x) - dxJ_a^\prime(x) $ has only real and simple zeros. According to Ismail and Muldoon \cite[Theorem 2.2]{Ismail-Muldoon}, the function $ z \mapsto (\alpha + \delta z) J_\nu(z) + (\beta + \tau z) J_\nu^\prime(z) $ has only real zeros if $ \beta = 0 $, $ \nu > -1 $ and $ \nu > -\alpha \tau^{-1} $. It can be shown by following the proof of \cite[Theorem 2.2]{Ismail-Muldoon} that the above result can be extended to the case when $ \beta = 0 $, $ \nu > -1 $ and $ \nu \geq -\alpha \tau^{-1} $. Now, choosing $ \alpha = ad, \delta = c, \beta = 0 $ and $ \tau = -d $, we get that $ x \mapsto (ad + cx)J_a(x) - dxJ_a^\prime(x) $ has only real zeros when $ a > -1 $.

Now, we shall prove that these zeros are simple. For this we denote the expression $ (ad + cx)J_a(x) - dxJ_a^\prime(x) $ by $ \varphi_{a,c,d}(x) $. The power series
\[
	\frac{2^a}{c}\Gamma(a+1) x^{-a-1}\varphi_{a,c,d}(x) = \sum_{n \geq 0} \frac{(-1)^n x^{2n}}{2^{2n} n! (a+1)_n} + \frac{d}{c} \sum_{n \geq 0} \frac{\Gamma(a+1)(-1)^n x^{2n+1}}{2^{2n-1}n!\Gamma(a+n+2)}
\]
is an entire function, which has growth order $ \frac{1}{2} $. Here $ (a)_n = a(a+1) \dots (a+n-1) $ is the so-called Pochhammer symbol. Thus, by using the Hadamard factorization theorem it can be shown that
\[
	\frac{2^a}{c} \Gamma(a+1)x^{-a-1}\varphi_{a,c,d}(x) = \prod_{k \geq 1} \bigg( 1 - \frac{x}{j_k(a,c,d)} \bigg),
\]
where $ j_k(a,c,d) $ denotes the $ n $-th zero of $ \varphi_{a,c,d}(x) $, see \eqref{Bessel_equation}-\eqref{Bessel_equation_solutions}. Taking the logarithmic derivative of both sides of the above relation we obtain
\[
	\frac{\varphi_{a,c,d}^\prime(x)}{\varphi_{a,c,d}(x)} = \frac{a+1}{x} + \sum_{k \geq 1} \frac{1}{x - j_k(a,c,d)}
\]
and
\[
	\bigg( \frac{\varphi_{a,c,d}^{\prime}(x)}{\varphi_{a,c,d}(x)} \bigg)^\prime = - \frac{a+1}{x^2} - \sum_{n \geq 1} \frac{1}{(x - j_k(a,c,d))^2} < 0
\]
for each $ a > -1 $ and each $ x $ which does not coincide with the zeros of $ \varphi_{a,c,d}(x) $. In the view of the limits
\[
	\lim_{x \searrow j_k(a,c,d)} \frac{\varphi_{a,c,d}^\prime(x)}{\varphi_{a,c,d}(x)} = \infty \quad \text{and} \quad \lim_{x \nearrow j_k(a,c,d)} \frac{\varphi_{a,c,d}^\prime(x)}{\varphi_{a,c,d}(x)} = -\infty
\]
and taking into account that the quotient $ \frac{\varphi_{a,c,d}^\prime(x)}{\varphi_{a,c,d}(x)} $ is strictly decreasing\footnote{It is important to mention here that the reality of the zeros of $ \varphi_{a,c,d}(x) $ and the infinite product actually imply that the function $ x \mapsto \frac{2^a}{c}\Gamma(a+1)x^{-a-1}\varphi_{a,c,d}(x) $ belongs to the Laguerre-P\'olya class of real entire functions, and hence satisfies the so-called Laguerre inequality. The decreasing property of the logarithmic derivative of $ \varphi_{a,c,d}(x) $ is a consequence of this inequality.} on each interval $ (j_k(a,c,d),j_{k+1}(a,c,d)) $, we obtain that the zeros of $ \varphi_{a,c,d}(x) $ and $ \varphi_{a,c,d}^\prime(x) $ are interlacing, which implies that the zeros of $ \varphi_{a,c,d}(x) $ are all simple.

\end{proof}

Now we are ready to prove our main theorem. \\

\noindent \textbf{Proof of Theorem \ref{main_theorem_1}.} We only prove the theorem for $ k = 0, 1, 2, \dots  $, the rest follows similarly, only using the asymptotic formula for the other side.  First we start with the proof of (a), then we use a similar argument to prove (b). \\

\textit{(a)} Let $ \arccos x_\nu = \pi \frac{p}{q} $ and $ n_l = q l + m $ for $ m \in \{ 0, 1, \dots, q-1 \} $ fixed. Because of the periodicity of the trigonometric terms in the equation
\begin{align*}
	\cos(n_l \arccos x_\nu + \varphi_\nu) a^{- \frac{\lambda_\nu - 1}{2}} J_{\frac{\lambda_\nu - 1}{2}}(a) + \sin(n_l \arccos x_\nu + \varphi_\nu) a^{- \frac{\lambda_\nu - 1}{2}} J_{\frac{\lambda_\nu + 1}{2}}(a) = 0,
\end{align*}
its solutions are the same for each $ l $. (The equations are not necessarily the same, but they are identical up to the multiplicative factor $ -1 $.) Therefore we define
\begin{align*}
	c_m = \cos(m \arccos x_\nu + \varphi_\nu), \quad d_m = \sin(m \arccos x_\nu + \varphi_\nu).
\end{align*}
Lemma \ref{Bessel_equation_simple_roots_lemma} says that the solutions  $ j_k(\lambda_\nu, c_m, d_m) $ are simple. Now, the scaled zeros $ a_{k,n_l}(x_0) $, which are the solutions of
\begin{align*}
	c_m a^{-\frac{\lambda_\nu - 1}{2}} J_{\frac{\lambda_\nu - 1}{2}}(a) + d_m  a^{-\frac{\lambda_\nu - 1}{2}} J_{\frac{\lambda_\nu + 1}{2}}(a) + O(n_l^{-1}) = 0,
\end{align*}
cannot be too close or too far from each other, that is, there is a constant $ C $ such that
\[
	\frac{1}{C} \leq a_{k+1,n_l}(x_0) - a_{k,n_l}(x_0) \leq C.
\]
This is a consequence of \cite[Theorem 1.1]{Varga} and the fact that the measure $ \mu $ is locally doubling around $ x_0 $. Because of these observations, it follows that by adding the error term $ O(n_l^{-1}) $ to the equation $ c_m a^{-\frac{\lambda_\nu - 1}{2}} J_{\frac{\lambda_\nu - 1}{2}}(a) + d_m  a^{-\frac{\lambda_\nu - 1}{2}} J_{\frac{\lambda_\nu + 1}{2}}(a) $, the existing roots are slightly perturbed, but we obtain no new roots otherwise. Thus, we have
\[
	a_{k,n_l}(x_0) = j_k\Big(\frac{\lambda_\nu - 1}{2}, c_m, d_m\Big) + o(1).
\]
It follows that
\begin{align*}
	\frac{n_l}{\sqrt{1 - x_\nu^2}}(x_\nu) & (x_{k+1,n_l}(x_\nu) - x_{k,n_l}(x_\nu)) \\
	& = j_{k+1}\Big(\frac{\lambda_\nu - 1}{2}, c_m, d_m\Big) - j_k\Big(\frac{\lambda_\nu - 1}{2}, c_m, d_m\Big) + o(1),
\end{align*}
which is what we needed to prove. \\

\textit{(b)} Now let $ c, d \in \mathbb{R} $ be arbitrary. The solutions of (\ref{Bessel_equation}) are left invariant if we multiply $ c $ and $ d $ with a fixed constant, therefore we can assume without the loss of generality that $ c^2 + d^2 = 1 $. Since $ \arccos x_\nu $ is not a rational multiple of $ \pi $, there exists a subsequence $ n_l $ such that
\[
	\big|\cos(n_l \arccos x_\nu + \varphi_\nu) - c\big| \leq \frac{1}{l}
\]
and
\[
	\big|\sin(n_l \arccos x_\nu + \varphi_\nu) - d\big| \leq \frac{1}{l}.
\]
Because of this and Lemma \ref{Bessel_equation_simple_roots_lemma}, the $ k $-th solution $ j_k(\frac{\lambda_\nu -1}{2}, c, d) $ of
\[
	c a^{-\frac{\lambda_\nu - 1}{2}} J_{\frac{\lambda_\nu - 1}{2}}(a) + d a^{-\frac{\lambda_\nu - 1}{2}} J_{\frac{\lambda_\nu + 1}{2}}(a) = 0
\]
 is of $ o(1) $ distance from the $ k $-th solution of
\[
	\cos(n_l \arccos x_\nu + \varphi_\nu) a^{-\frac{\lambda_\nu - 1}{2}} J_{\frac{\lambda_\nu - 1}{2}}(a) + \sin(n_l \arccos x_\nu + \varphi_\nu) a^{-\frac{\lambda_\nu - 1}{2}} J_{\frac{\lambda_\nu + 1}{2}}(a) = 0.
\]
Following the argument in the proof of (a), we easily obtain that
\begin{align*}
	\frac{n_l}{\sqrt{1 - x_\nu^2}}(x_\nu) & (x_{k+1,n_l}(x_\nu) - x_{k,n_l}(x_\nu)) \\
	& = j_{k+1}\Big(\frac{\lambda_\nu -1}{2}, c, d\Big) - j_k\Big(\frac{\lambda_\nu -1}{2}, c, d\Big) + o(1),
\end{align*}
which completes the proof of our main theorem. \begin{flushright} $ \Box $ \end{flushright}

\section{Zeros of the function $ c J_{a}(x) + d J_{a+1}(x) $}\label{Bessel_linear_combination_zeros}

In this section we study the zeros of the special function
\[
	\psi_{a,c,d}(x) = c J_{a}(x) + d J_{a+1}(x), 
\]
in order to prove estimates regarding the limiting behavior of the zeros. Recall that the positive roots of the equation $ \psi_{a,c,d}(x) $ are
\[
	0 < j_1(a,c,d) < j_2(a,c,d) < \dots,
\]
see (\ref{Bessel_equation})-(\ref{Bessel_equation_solutions}). Our strategy is to find a second order linear differential equation which is satisfied by $ \psi_{a,c,d}(x) $, write it in a Sturm-Liouville normal form, then use the Sturm comparison and convexity theorems, see \cite[Theorem 1]{Deano-Gil-Segura}). First we prove Theorem \ref{zero_convexity_theorem}, which is an application of the Sturm convexity theorem. \\

\noindent \textbf{Proof of Theorem \ref{zero_convexity_theorem}.} Using that $ J_a(x) $ satisfies the Bessel differential equation
\[
	x^2 y^{\prime \prime}(x) + xy^\prime(x) + (x^2 - a^2)y(x) = 0,
\]
it is possible to show that the special function $\psi_{a,c,d}(x)=cJ_a(x)+dJ_{a+1}(x) $ satisfies the following second-order homogeneous ordinary differential equation\footnote{The author \'A. Baricz is very grateful to Christoph Koutchan for deducing this differential equation with his Holonomic Functions Package.}
\begin{equation}\label{eqCK}
	u(x)y''(x)+v(x)y'(x)+w(x)y(x)=0,
\end{equation}
where
\[
	u(x)=(2a+1)cdx^2+(c^2+d^2)x^3,\ \ v(x)=2(2a+1)cdx+(c^2+d^2)x^2,
\]
\[
	w(x)=-a(a+1)(2a+1)cd-(a^2c^2+d^2+2ad^2+a^2d^2)x+(2a+1)cdx^2+(c^2+d^2)x^3.
\]
Equation \eqref{eqCK} can be rewritten as
\[
	y''(x)+B(x)y'(x)+A(x)y(x)=0,
\]
where $ B(x)= v(x) u(x)^{-1} $ and $ A(x)=w(x) u(x)^{-1} $, and it can be shown that the function
$ z(x)=\exp \big( \frac{1}{2}\int_{0}^{x}B(s)ds \big) y(x) $ satisfies the second-order differential equation
\[
	z''(x)+C(x)z(x)=0,\ \ \ C(x)=A(x)-\frac{1}{2}B'(x)-\frac{1}{4}B^2(x),
\]
which is in normal form and it is suitable for the technique of the Sturm theory. We have to differentiate between the cases $ a \in (-\frac{1}{2}, 0] $, $ a \in (0,\frac{1}{2}) $ and $ a \in [\frac{1}{2}, \infty) $. \\

\textit{(i) The case when $ a \in [\frac{1}{2}, \infty) $ and $ c $ has the same sign as $ d $.} In this case $ C(x) $ can be written as
\begin{equation}\label{C_definition}
	C(x)=\frac{a_0x^4+a_1x^3+a_2x^2+a_3x+a_4}{4x^2((c^2+d^2)x+(2a+1)cd)^2},
\end{equation}
where
\[
	a_0=4(c^2+d^2)^2, \ \ a_1=8(2a+1)cd(c^2+d^2),
\]
\[
	a_2=(2a+1)(d^2-c^2)(2a(c^2-d^2)-c^2-3d^2),
\]
\[
	a_3=-4cd(2a+1)(2a^2(c^2+d^2)+a(c^2+3d^2)-c^2),
\]
\[
	a_4=-4ac^2d^2(a+1)(2a+1)^2.
\]
Now, if we compute the derivative of $ C(x) $ we obtain
\[
	C'(x)=\frac{(2a+1)(b_0x^3+b_1x^2+b_2x+b_3)}{2x^3((c^2+d^2)x+(2a+1)cd)^3},
\]
where
\[
	b_0=(c^2+d^2)^2(2a(c^2+d^2)-c^2+3d^2),
\]
\[
	b_1=6cd(c^2+d^2)(2a^2(c^2+d^2)+a(c^2+3d^2)-c^2),
\]
\[
	b_2=2(2a+1)c^2d^2(6a^2(c^2+d^2)+a(5c^2+7d^2)-c^2),
\]
\[
	b_3=4a(a+1)(2a+1)^2c^3d^3.
\]
Now, since $ c $ and $ d $ have the same sign and $ a\geq\frac{1}{2} $, the coefficients $ b_0, b_1, b_2 $ and $ b_3 $ are all positive, and consequently under the same assumptions the function $ C $ is strictly increasing on $ (0,\infty) $. The above result together with the Sturm comparison and convexity theorem (see for example \cite[Theorem 1]{Deano-Gil-Segura}) imply that when $ c $ and $ d $ are real with the same sign, $ a \geq \frac{1}{2} $ and $ k \in \mathbb{N} $ we have \eqref{zero_convexity_1}. \\

\textit{(ii) The case when $a\in\left(0,\frac{1}{2}\right)$ and $ c $ has the same sign as $ d $.} It is straightforward to check that given a function $y(x),$ which is a solution of
\begin{equation}\label{eqnf}y''(x)+B(x)y'(x)+A(x)y(x)=0,\end{equation}
then the function $Y(z),$ with $Y(z(x))$ given by
$$Y(z(x))=\sqrt{z'(x)}\exp \Big( \frac{1}{2}\int_{0}^{x} B(s) ds \Big) y(x)$$
satisfies the equation in normal form
\begin{equation}\label{Liouville_transform_eq}
	\ddot{Y}(z)+\Omega(z)Y(z)=0.
\end{equation}
Here the dots mean differentiation with respect to $z$ and
\[
	\Omega(z) = \dot{x}(z)^2 C(x(z)) - \dot{x}(z)^{\frac{1}{2}} \frac{d^2}{dz^2} \dot{x}(z)^{-\frac{1}{2}}.
\]
$ \Omega $ can also be written as a function of $ x $, that is we have
$$\Omega(x)=\Omega(z(x))=\frac{1}{d^2(x)}\left[C(x)+\frac{3}{4}\left(\frac{d'(x)}{d(x)}\right)^2-\frac{d''(x)}{2d(x)}\right],$$
where $d(x)=z'(x).$ See \cite{Deano-Gil-Segura} for more details on this Liouville transformation, especially equation (13) and the remark after it. It is important to note that if $ B(x) $ is continuous, which is satisfied for $ a > -\frac{1}{2} $, the transformed function $ Y(x) = Y(z(x)) $ has the same zeros as $ y(x) $.

Now, we consider the change of variable $d(x)=(xu(x))^{-1}$ in the Liouville transformation. In this case the singularities of $\Omega(x)$ will disappear and $\Omega(x)$ will be a
polynomial. Namely, we have that
$$\Omega(x)=d_0x^8+d_1x^7+d_2x^6+d_3x^5+d_4x^4,$$
where
$$d_0=(c^2+d^2)^2,\ \ d_1=2(2a+1)cd(c^2+d^2),$$
$$d_2=-\frac{1}{4}\left(c^4(4a^2-9)+c^2d^2(-8a^2-8a-18)+d^4(4a^2+8a-5)\right),$$
$$d_3=-cd\left(c^2(4a^3+4a^2-7a-4)+d^2(4a^3+8a^2-3a-3)\right),$$
$$d_4=-\frac{1}{4}c^2d^2(16a^4+32a^3+8a^2-8a-3).$$
If $c$ and $d$ have the same sign and $a\in\left(0,\frac{1}{2}\right),$ then the coefficients $d_0,$ $d_1,$ $d_2,$ $d_3$ and $d_4$ are all positive, and consequently under the same assumptions the function $\Omega$ is strictly increasing on $(0,\infty).$ Applying once more the Sturm comparison and convexity theorem (see
\cite[Theorem 1]{Deano-Gil-Segura}) we conclude the convexity of the zeros $z(j_k(a,c,d)).$ 
Since in this case
$$z(x)=-\frac{(c^2+d^2)^2}{c^3d^3(2a+1)^3}\ln\left(c^2+d^2+\frac{cd(2a+1)}{x}\right)+\frac{c^2+d^2}{c^2d^2(2a+1)^2x}-\frac{1}{2cd(2a+1)x^2},$$
the convexity of the sequence $\left\{z(j_k(a,c,d))\right\}_{k\geq1}$ is equivalent to the inequality
$$\frac{\left(1+\frac{\sigma}{j_{k+2}(a,c,d)}\right)\left(1+\frac{\sigma}{j_{k}(a,c,d)}\right)}{\left(1+\frac{\sigma}{j_{k+1}(a,c,d)}\right)^2}
>e^{\sigma\Delta^2 j_k(a,c,d)-\sigma^2\Delta^2 j_k^2(a,c,d)},$$
where
$$\sigma=\frac{(2a+1)cd}{c^2+d^2}$$
and
$$\Delta^2j_k(a,c,d)=\frac{1}{j_{k+2}(a,c,d)}+\frac{1}{j_{k}(a,c,d)}-\frac{2}{j_{k+1}(a,c,d)},$$
which is \eqref{zero_convexity_2}. \\

\textit{(iii) The case when $a\in\left(-\frac{1}{2},0\right]$ and $ c $ has the same sign as $ d $.} Now, we consider the change of variable $d(x)=u(x)^{-1}.$ In this case the singularities of $\Omega(x)$ will disappear and $\Omega(x)$ will be a
polynomial. Namely, we have that
$$\Omega(x)=c_0x^6+c_1x^5+c_2x^4+c_3x^3+c_4x^2,$$
where
$$c_0=(c^2+d^2)^2,\ \ c_1=2(2a+1)cd(c^2+d^2),$$
$$c_2=-\left(c^4(a^2-1)+c^2d^2(-2a^2-2a-2)+d^4(a^2+2a)\right),$$
$$c_3=-cd\left(c^2(4a^3+4a^2-3a-2)+d^2(4a^3+8a^2+a-1)\right),$$
$$c_4=-ac^2d^2(4a^3+8a^2+5a+1).$$
If $c$ and $d$ have the same sign and $a\in\left(-\frac{1}{2},0\right],$ then the coefficients $c_0,$ $c_1,$ $c_2,$ $c_3$ and $c_4$ are all positive, and consequently under the same assumptions the function $\Omega$ is strictly increasing on $(0,\infty).$ Applying again the Sturm comparison and convexity theorem (see
\cite[Theorem 1]{Deano-Gil-Segura}) we conclude the convexity of the zeros $z(j_k(a,c,d)),$ that is, we have
$$z(j_{k+2}(a,c,d)-2z(j_{k+1}(a,c,d))+z(j_{k}(a,c,d))<0.$$
Since in this case
$$z(x)=\frac{c^2+d^2}{c^2d^2(2a+1)^2}\ln\left(c^2+d^2+\frac{cd(2a+1)}{x}\right)-\frac{1}{cd(2a+1)x},$$
the above inequality becomes
$$\frac{\left(1+\frac{\sigma}{j_{k+2}(a,c,d)}\right)\left(1+\frac{\sigma}{j_{k}(a,c,d)}\right)}{\left(1+\frac{\sigma}{j_{k+1}(a,c,d)}\right)^2}
<e^{\sigma\Delta^2 j_k(a,c,d)},$$
which is \eqref{zero_convexity_3}. \\

\textit{(iv) The case when $ a\in\left(-1,-\frac{1}{2}\right)$, $ c $ has the opposite sign as $ d $ and $ d^2 \geq c^2 $.} We consider the change of variable $d(x)=x^3 u(x)^{-1}$ in the Liouville transformation. 
In this case $\Omega(x)$ becomes
$$\Omega(x)=e_0+e_1\frac{1}{x}+e_2\frac{1}{x^2}+e_3\frac{1}{x^3}+e_4\frac{1}{x^4},$$
where
$$e_0=(c^2+d^2)^2,\ \ e_1=2(2a+1)cd(c^2+d^2),$$
$$e_2=-\frac{1}{4}\left(4a^2(c^2-d^2)^2+8ad^2(d^2-c^2)-c^4-2c^2d^2+3d^4\right),$$
$$e_3=-cd(2a+1)\left(2a^2(c^2+d^2)+a(c^2+3d^2)-2c^2-d^2\right),$$
$$e_4=-\frac{1}{4}c^2d^2(2a+1)^2(4a^2+4a-3).$$
Computing the derivative of $\Omega(x)$ we obtain
$$\Omega'(x)=f_0\frac{1}{x^2}+f_1\frac{1}{x^3}+f_2\frac{1}{x^4}+f_3\frac{1}{x^5},$$
where
$$f_0=-2cd(2a+1)(c^2+d^2),$$
$$f_1=\frac{1}{2}\left(c^4(4a^2-1)+c^2d^2(-8a^2-8a-2)+d^4(4a^2+8a+3)\right),$$
$$f_2=3cd(2a+1)\left(c^2(2a^2+a-2)+d^2(2a^2+3a-1)\right),$$
$$f_3=-c^2d^2(2a+1)^2(4a^2+4a-3).$$

If $d^2\geq c^2,$ $cd<0$ and $a\in\left(-1,-\frac{1}{2}\right),$ then the coefficients $f_0,$ $f_1,$ $f_2$ and $f_3$ are all negative, and consequently under the same assumptions the function $\Omega$ is strictly decreasing on $(0,\infty).$ Since $\Omega(x)\to (c^2+d^2)^2$ as $x\to\infty,$ it follows that under the same assumptions $\Omega(x)>(c^2+d^2)^2$ for $x>0.$ Applying again the Sturm comparison and convexity theorem (see \cite[Theorem 1]{Deano-Gil-Segura}) we conclude the concavity of the zeros $z(j_k(a,c,d))$ and the inequality
$$z(j_{k+1}(a,c,d))-z(j_k(a,c,d))<\frac{\pi}{c^2+d^2}.$$
These are equivalent to
$$\frac{(\sigma+j_{k+2}(a,c,d))(\sigma+j_{k}(a,c,d))}{(\sigma+j_{k+1}(a,c,d))^2}<e^{\frac{1}{\sigma}\delta^2j_{k}(a,c,d)},$$
which is \eqref{zero_convexity_4}, where $ \delta^2 j_k(a,c,d) $ is defined by
\begin{equation}\label{zero_comparison_delta}
	\delta^2 j_k(a,c,d)={j_{k+2}(a,c,d)}-2{j_{k+1}(a,c,d)}+j_k(a,c,d),
\end{equation}
and
$$j_{k+1}(a,c,d)-j_{k}(a,c,d)-\sigma\ln\frac{\sigma+j_{k+1}(a,c,d)}{\sigma+j_{k}(a,c,d)}<\pi,$$
which is \eqref{zero_comparison_case_3}. \\
\begin{flushright} $ \Box $ \end{flushright}

For certain subset of parameters, it is possible to provide estimates for the distance of two neighbouring roots. Now we prove Theorem \ref{zero_comparison_theorem}, which is ultimately given by an application of the Sturm comparison theorem.  \\

\noindent \textbf{Proof of Theorem \ref{zero_comparison_theorem}.} \\
\textit{(i) The case when $a\in\left(\frac{1}{2}, \infty \right)$ and $ c $ has the same sign as $ d $.} Observe that $C(x)$ in \eqref{C_definition} can be rewritten as
$$C(x)=\frac{a_0x^4+a_1x^3+a_2x^2+a_3x+a_4}{a_0x^4+a_1x^3+q_2x^2}=1+\frac{1}{x^2}\frac{(a_2-q_2)x^2+a_3x+a_4}{a_0x^2+a_1x+q_2},$$
where $q_2=4(2a+1)^2c^2d^2$ and $$a_2-q_2=-(2a+1)((2a-1)c^4+2(2a+1)c^2d^2+(2a+3)d^4).$$
Thus, if $a\geq \frac{1}{2}$ and $c,$ $d$ have the same sign, then the coefficients $a_2-q_2,$ $a_3$ and $a_4$ are strictly negative, while the
coefficients $a_0,$ $a_1$ and $q_2$ are strictly positive. This implies that under the same conditions we have that $C(x)<1$ for $x>0.$ An alternative way to prove this is to recall that the function $C$ is strictly increasing on $(0,\infty),$ tends to $1$ when $x$ tends to infinity, and tends to $-\infty$ when $x$ tends to zero, and its graph is not crossing its horizontal asymptote $y=1$. Applying the Sturm comparison and convexity theorem (see \cite[Theorem 1]{Deano-Gil-Segura}) we conclude that
$$j_{k+1}(a,c,d)-j_k(a,c,d)>\pi$$ for each $k\in\mathbb{N},$ $a\geq \frac{1}{2}$ and $c,$ $d$ real with the same sign. \\

\textit{(ii) The case when $a \geq 0$, $ c $ has the same sign as $ d $ and $ d^2 \geq c^2 $.} We consider the change of variable $d(x)=u(x)x^{-3}.$ In this case $\Omega(x)$ and $\Omega'(x)$ become
$$\Omega(x)=\frac{4u(x)-(2a+1)(2a^2(c^2+d^2)-c^2+3d^2)x-(2a+1)^3cd}{4((c^2+d^2)x+(2a+1)cd)^3},$$
$$\Omega'(x)=\frac{2a+1}{2((c^2+d^2)x+(2a+1)cd)^4}\Theta(x),$$
where
\begin{align*}\Theta(x)=4cd(c^2+d^2)&x^2+((2a+1)(c^4+6c^2d^2+d^4)+2(d^4-c^4))x\\&+2cd((2a^2+3a+1)c^2+(2a^2+a)d^2)\end{align*}
If $c$ and $d$ have the same sign, $d^2\geq c^2$ and $a\geq0,$ then $\Theta(x)>0$ for each $x>0,$ and thus the function $\Omega$ is monotone increasing on
$(0,\infty).$ Since $\Omega(x)$ tends to $(c^2+d^2)^{-2}$ as $x\to\infty,$ it follows that $\Omega(x)<(c^2+d^2)^{-2}.$ Applying again the Sturm comparison and convexity theorem (see \cite[Theorem 1]{Deano-Gil-Segura}) we conclude the inequalities 
$$z(j_{k+1}(a,c,d))-z(j_k(a,c,d))>(c^2+d^2)\pi$$
and 
$$z(j_{k+2}(a,c,d)-2z(j_{k+1}(a,c,d))+z(j_{k}(a,c,d))<0.$$
Since in this case $z(x)=(2a+1)cd\ln(x)+(c^2+d^2)x,$ the above inequalities are equivalent to
$$\sigma\ln\frac{j_{k+1}(a,c,d)}{j_k(a,c,d)}+{j_{k+1}(a,c,d)}-{j_k(a,c,d)}>\pi$$
and 
$$\sigma\ln\frac{j_{k+2}(a,c,d)j_k(a,c,d)}{j_{k+1}^2(a,c,d)}+\delta^2 j_k(a,c,d)<0,$$
where $ \delta^2 j_k(a,c,d) $ is defined by \eqref{zero_comparison_delta}, which is \eqref{zero_comparison_case_2}. \\

\textit{(iii) The case when $a\in\left(-1,-\frac{1}{2}\right)$, $ c $ has the opposite sign as $ d $ and $ d^2 \geq c^2 $.} For the proof of \eqref{zero_comparison_case_3}, see the proof of \textit{(iv)} in Theorem \ref{zero_convexity_theorem}. As mentioned in the proof, \eqref{zero_comparison_case_3} is equivalent to \eqref{zero_convexity_4}.
\begin{flushright} $ \Box $ \end{flushright}

Although the Liouville transformation combined with Sturm theory yields different inequalities for different domains of parameters, asymptotic results can be established for the complete range of parameters. \\

\noindent \textbf{Proof of Theorem \ref{Bessel_zero_asymptotic_theorem}.} As it was observed in the proof of Theorem \ref{zero_comparison_theorem} (i), the $ C(x) $ defined by \eqref{C_definition} can be rewritten as
\[
	C(x) = 1 + O(x^{-2}).
\]
This means that if $ \varepsilon > 0 $ is arbitrary, we can find an $ x_0 $ such that
\[
	\frac{1}{1 + \varepsilon} \leq C(x) \leq 1 + \varepsilon, \quad x > x_0
\]
holds. Applying the Sturm comparison theorem for $ (x_0, \infty) $, see \cite[Theorem 1]{Deano-Gil-Segura}, we obtain that
\[
	\frac{\pi}{\sqrt{1 + \varepsilon}} < j_{k+1}(a,c,d) - j_k(a,c,d) < \pi \sqrt{1 + \varepsilon}
\]
holds for all $ k $ for which $ j_k(a,c,d) > x_0 $ is satisfied. Since $ \lim_{k \to \infty} j_k(a,c,d) = \infty $, such $ k $ exists, and since $ \varepsilon > 0 $ was arbitrary, \eqref{Bessel_zero_asymptotic_equation} follows.
\begin{flushright} $ \Box $ \end{flushright}

\vspace{1cm}
\'Arp\'ad Baricz \\ \textsc{Institute of Applied Mathematics, \'Obuda University, 1034 Budapest, Hungary }\\ \textsc{Department of Economics, Babe\c{s}-Bolyai University, 400591 Cluj-Napoca, Romania} \\
email: bariczocsi@yahoo.com \\
URL: \textsc{https://sites.google.com/site/bariczocsi/} \\[0.5cm]

\noindent Tivadar Danka \\
\textsc{Bolyai Institute, University of Szeged, 6720 Szeged, Hungary} \\
email: tdanka@math.u-szeged.hu \\
URL: \textsc{http://www.tivadardanka.com}
\end{document}